\documentclass[11pt]{article}

\usepackage[a4paper, total={6in, 9in}]{geometry}

\usepackage{graphicx}%
\usepackage{multirow}%
\usepackage{amsmath,amssymb,amsfonts}%
\usepackage{amsthm}%
\usepackage{mathrsfs}%
\usepackage[title]{appendix}%
\usepackage{xcolor}%
\usepackage{textcomp}%
\usepackage{manyfoot}%
\usepackage{booktabs}%
\usepackage{algorithm}%
\usepackage{algorithmicx}%
\usepackage{algpseudocode}%
\usepackage{listings}%

\usepackage{hyperref}
\newtheorem{theorem}{Theorem}[section]
\newtheorem{proposition}[theorem]{Proposition}
\newtheorem{lemma}[theorem]{Lemma}
\newtheorem{corollary}[theorem]{Corollary}

\newtheorem{assumption}[theorem]{Assumption}
\usepackage{authblk}
\usepackage[square, numbers, sort]{natbib}
\usepackage{graphicx}
\usepackage{orcidlink}

\hypersetup{
    colorlinks=true,
    linkcolor=blue,
    filecolor=blue,      
    urlcolor=cyan,
    citecolor = blue
}

\DeclareMathOperator{\supp}{supp}
\DeclareMathOperator{\vspan}{span}
\DeclareMathOperator{\dist}{dist}

\title{Directional polynomial wavelets on spheres}

\date{}

\author{Frederic Schoppert \orcidlink{0000-0002-8682-3723} \\ \href{mailto:f.schoppert@uni-luebeck.de}{f.schoppert@uni-luebeck.de} }

\affil{Institute of Mathematics, University of L{\"u}beck, Ratzeburger Allee 160, 23562 L{\"u}beck, Germany}

\begin{document}

\maketitle

\begin{abstract}
In this article, we construct discrete tight frames for $L^2(\mathbb{S}^{d-1})$, $d\geq3$, which consist of localized polynomial wavelets with adjustable degrees of directionality. In contrast to the well studied isotropic case, these systems are well suited for the direction sensitive analysis of anisotropic features such as edges. The price paid for this is the fact that at each scale the wavelet transform lives on the rotation group $SO(d)$, and not on $\mathbb{S}^{d-1}$ as in the zonal setting. Thus, the standard approach of building discrete frames by sampling the continuous wavelet transform requires a significantly larger amount of sample points. However, by keeping the directionality limited, this number can be greatly reduced to the point where it is comparable to the number of samples needed in the isotropic case. Moreover, the limited directionality is reflected in the wavelets being steerable and their great localization in space leads to a fast convergence of the wavelet expansion in the spaces $L^p(\mathbb{S}^{d-1})$, $1\leq p \leq \infty$.
\end{abstract}

\paragraph{Keywords.}
Wavelets on the sphere, localized polynomial frames, directional wavelets

\paragraph{Mathematics Subject Classification.} 42C15, 42C40, 65T60

\section{Introduction}
Localized polynomial frames for $L^2(\mathbb{S}^{d-1})$, $d\geq3$, consisting of zonal (i.e.\ isotropic) analysis functions have been extensively studied in the literature (see e.g.\ \citep{bib3, bib11, bib15, bib29, bib30, bib35, bib36, bib37}). A reoccurring central part in these investigation is the pursuit and utilization of good localization bounds which, on the one hand, confirm the systems ability to perform precise position-frequency analyses and, on the other hand, allow for good approximation properties in more general spaces such as $L^p(\mathbb{S}^{d-1})$, $1\leq p \leq \infty$. The advantage of considering zonal frame functions lies in their simple structure. Indeed, isotropic functions on the sphere correspond to one-dimensional algebraic polynomials and thus their localization properties can be studied in an uni-variate setting. In general, working with polynomial frames has far-reaching computational advantages. For example, it enables the exact reconstruction of band-limited signals in practice. Also, in this case the frame coefficients are finite linear combinations of the classical Fourier coefficients.

As a more recent development, localized polynomial frames which are also directional have been introduced on the two-dimensional unit sphere $\mathbb{S}^2$ to allow for an appropriate position-based analysis of aniso\-tropic features \citep{bib21, bib1, bib9, bib10}. Indeed, we previously derived new and improved localization bounds for these systems and showed that local geometric properties of edges in signals are accurately displayed by the correspon\-ding frame coefficients in an arbitrary small neighborhood \citep{bib26, bib31}. 

In this paper, we introduce and investigate a large class of localized directional polynomial frames for $L^2(\mathbb{S}^{d-1})$, $d\geq 3$. Our systems result from a construction which generalizes the approach in \citep{bib21, bib9, bib10} to higher dimensions and most of the frames mentioned above are included as special cases. The analysis functions, or wavelets, are designed to be of adjustable, but limited, directional sensitivity which makes them steerable. Also, as usual for polynomial systems, we obtain fully discretized frames by sampling the corresponding continuous wavelet transform. In the isotropic case, the latter produces polynomials on $\mathbb{S}^{d-1}$ and a discrete frame can be obtained by choosing a suitable sequence of quadrature points on the sphere. In our general setting, the continuous wavelet transform lives on the special orthogonal group $SO(d)$ and therefore requires a much higher amount of samples. However, by utilizing the fact that our wavelets are limited in their directionality, this number can be reduced significantly. Indeed, the amount of sample points needed at each scale is similar to the isotropic case. A central result of this paper is the localization bound for the frame functions given in \autoref{theorem1}. It generalizes the well known estimate for isotropic systems and enables us to show that, under mild assumptions on the discretization grid, the frame expansion converges fast in $L^p(\mathbb{S}^{d-1})$, $1 \leq p \leq \infty$. Finally, we explicitly discuss special cases where the wavelet functions are highly symmetric and exhibit an optimal directionality.

We point out that for $d=3$ most of the results presented in this article were previously established in \citep{bib21, bib9, bib10}. This includes the construction of fully discretized directional wavelet systems as well as corresponding properties such as steerability, optimal (but limited) directional sensitivity, reduced computational cost, symmetries and spatial localization. With this current article, we add to this a discussion on the convergence properties of the directional wavelet expansion in $L^p(\mathbb{S}^2)$, $1 \leq p \leq \infty$,  which has not been considered in the literature. Additionally, our localization bound in \autoref{theorem1} slightly improves the existing estimate from \citep{bib9}.

The remainder of this article is organized as follows. In \autoref{sec2}, we give a brief summary of Fourier analysis on spheres and rotation groups. Particularly, we highlight several key properties of polynomials on $\mathbb{S}^{d-1}$ and $SO(d)$, as well as their connections. Moreover, an important class of polynomial subspaces on $SO(d)$ is introduced and investigated. In \autoref{sec3}, fully discretized polynomial frames for $L^2(\mathbb{S}^{d-1})$ are constructed. Properties regarding the localization and directionality of these systems are investigated in \autoref{sec4}. In \autoref{sec5}, we discuss approximation in $L^p(\mathbb{S}^{d-1})$ and, finally, in \autoref{sec6}, we propose special cases of highly symmetric wavelets with an optimal directionality.

\section{Preliminaries} \label{sec2}
In this section, we summarize some basic concepts of harmonic analysis on spheres and rotation groups, which will be of fundamental importance in our constructions. Indeed, our frame elements, the directional wavelets, are designed to be polynomials on $\mathbb{S}^{d-1}$ and the corresponding wavelet transform lives on $SO(d)$. Thus, the calculations involved heavily rely on Fourier methods for both of these spaces. The majority of formulas given in this section can be found in classical literature such as \citep[Chapter~9]{bib32} and are stated without further comment. 

For $d\geq 3$, we consider the euclidean space $\mathbb{R}^d$ equipped with the inner product $\langle x, y \rangle = x^\top y$ and the induced norm $\| \cdot \|$. Its canonical basis vectors will be denoted by $e^j=(\delta_{i, j})_{i=1}^d$, $j=1, ..., d$, where
\begin{equation*}
\delta_{i, j}= \begin{cases}
1, \quad & i=j,\\
0,  & i \neq j.
\end{cases}
\end{equation*}
The unit sphere $\mathbb{S}^{d-1} = \{ x \in \mathbb{R}^d : \| x \| = 1  \} $ can be parameterized in terms of spherical coordinates $\theta_1 \in [0, 2\pi)$, $\theta_2, ..., \theta_{d-1}\in [0, \pi]$ via
\begin{equation*}
\eta(\theta_1, \theta_2, ..., \theta_{d-1}) = \begin{pmatrix}
\sin \theta_{d-1}\, ... \, \sin \theta_2 \, \sin \theta_1 \\
\sin \theta_{d-1}\, ... \, \sin \theta_2 \, \cos \theta_1 \\
\sin \theta_{d-1}\,  ...\,  \sin \theta_3 \, \cos \theta_2 \\
\vdots \\
\sin \theta_{d-1} \, \cos \theta_{d-2} \\
\cos \theta_{d-1}
\end{pmatrix}.
\end{equation*}
A metric on $\mathbb{S}^{d-1}$ is given by the geodesic distance
\begin{equation*}
\dist(\eta, \nu) = \arccos(\langle \eta, \nu \rangle), \quad \eta, \nu \in \mathbb{S}^{d-1}.
\end{equation*}
In particular, $\dist(\eta(\theta_1, ..., \theta_{d-1}), e^d) = \theta_{d-1}$. By $\omega_{d-1}$ we will denote the usual rotation-invariant surface measure on $\mathbb{S}^{d-1}$ which is normalized such that
\begin{equation*}
\int_{\mathbb{S}^{d-1}} \mathrm{d}\omega_{d-1} = 1. 
\end{equation*}
In spherical coordinates,
\begin{equation*}
\mathrm{d}\omega_{d-1} = \frac{\Gamma(\frac{d}{2})}{2 \pi^{d/2}} \sin^{d-2}\theta_{d-1} \, ...\,\sin \theta_2 \, \mathrm{d}\theta_1 \, ...\, \mathrm{d}\theta_{d-1}.
\end{equation*}
With respect to the surface measure $\omega_{d-1}$, we define the spaces 
\begin{equation*}
L^p(\mathbb{S}^{d-1}) = \{ f \colon \mathbb{S}^{d-1} \rightarrow \mathbb{C} \mid f  \text{ measurable and } \|f\|_{L^p(\mathbb{S}^{d-1})} < \infty \},
\end{equation*}
where
\begin{equation*}
\| f\|_{L^p(\mathbb{S}^{d-1})} =
 \left( \int_{\mathbb{S}^{d-1}} \lvert f \rvert^p \, \mathrm{d}\omega_{d-1} \right)^{1/p}, \quad 0 <p < \infty.
\end{equation*}
For $p=\infty$, we consider $L^\infty(\mathbb{S}^{d-1})$ to be the space of continuous functions on $\mathbb{S}^{d-1}$ with the supremum norm. The Hilbert space $L^2(\mathbb{S}^{d-1})$ is equipped with the inner product
\begin{equation*}
\langle f_1, f_2 \rangle_{\mathbb{S}^{d-1}} = \int_{\mathbb{S}^{d-1}} f_1 \, \overline{f_2} \, \mathrm{d}\omega_{d-1}.
\end{equation*}
An explicit orthonormal basis for $L^2(\mathbb{S}^{d-1})$ is given by the spherical harmonics
\begin{align}\label{y_n^k def}
Y_k^{d,n}(\theta_1, ..., \theta_{d-1})= A_k^n \prod_{j=0}^{d-3} C_{ k_j  - \lvert k_{j+1}\rvert}^{\frac{d-j-2}{2}+\lvert k_{j+1}\rvert}(\cos \theta_{d-j-1}) \, \sin^{\lvert k_{j+1}\rvert}(\theta_{d-j-1}) \, \mathrm{e}^{\mathrm{i}k_{d-2}\theta_1},
\end{align}
where  $k_0=n \in \mathbb{N}_0$ and $k=(k_1, ..., k_{d-2})\in \mathcal{I}_n^d$ with
\begin{equation*}
\mathcal{I}_n^d = \{ (k_1, ..., k_{d-2}) \in \mathbb{N}_0^{d-3}\times \mathbb{Z} : n\geq k_1 \geq ...\geq k_{d-3}\geq \lvert k_{d-2}\rvert \}.
\end{equation*}
The one-dimensional functions $C_m^\lambda$ occurring in \eqref{y_n^k def} are the Gegenbauer polynomials which can be defined through
\begin{equation*}
C_m^\lambda(t) = \frac{(-1)^m \Gamma(\lambda+1/2) \Gamma(m+2\lambda)}{2^m\Gamma(m+\lambda+1/2) m!} (1-t^2)^{1/2-\lambda} \frac{\mathrm{d}^{m}}{\mathrm{d}t^{m}}(1-t^2)^{m+\lambda-1/2}.
\end{equation*}
Also, the normalization factor $A_k^n >0$ satisfies
\begin{align}\label{eq10}
(A_k^n)^2 = \frac{2^{(d-4)(d-2)}}{\Gamma\!\left(\frac{d}{2}\right)} \prod_{j=0}^{d-3}\frac{2^{2\lvert k_{j+1}\rvert-j}(k_j-\lvert k_{j+1}\rvert)!(2k_j+d-j-2)\Gamma^2(\frac{d-j-2}{2}+\lvert k_{j+1}\rvert)}{\sqrt{\pi}\Gamma(k_j+\lvert k_{j+1}\rvert+d-j-2)}.
\end{align}

The space $\Pi_N(\mathbb{S}^{d-1}) = \vspan \{ Y_k^{d, n} : n \leq N,\; k \in \mathcal{I}_n^d \}$, consisting of all spherical polynomials of degree $N$ or less, can be decomposed into the spherical harmonic subspaces $\mathcal{H}_n^d = \vspan \{ Y_k^{d, n} : k \in \mathcal{I}_n^d \}$. Namely,
\begin{equation*}
\Pi_N(\mathbb{S}^{d-1}) = \bigoplus_{n=0}^N \mathcal{H}_n^d.
\end{equation*}
Concerning the dimensions of these spaces, we have
\begin{equation*}
\dim \mathcal{H}_{n}^d = \frac{(2n+d-2) (n+d-3)!}{(d-2)! n!}
\end{equation*} 
and
\begin{equation*}
\dim \Pi_N(\mathbb{S}^{d-1}) = \frac{(2N+d-1)(N+d-2)!}{(d-1)!N!}.
\end{equation*}
The well known addition theorem  for spherical harmonics reads
\begin{equation}\label{addition theorem}
\sum_{k \in \mathcal{I}_n^d} Y_k^{d, n}(\eta) \, \overline{Y_k^{d, n}(\nu)} = \frac{2n+d-2}{d-2}\,  C_n^{\frac{d-2}{2}}(\langle \eta, \nu \rangle), \quad \eta, \nu \in \mathbb{S}^{d-1}.
\end{equation}
Lastly, we note that polynomials on the sphere satisfy the product property
\begin{equation*}
f_1 f_2 \in \Pi_{2N}(\mathbb{S}^{d-1}) \quad \text{if } f_1, f_2 \in \Pi_N(\mathbb{S}^{d-1}).
\end{equation*}
I.e., $\Pi_N(\mathbb{S}^{d-1}) \cdot \Pi_N(\mathbb{S}^{d-1}) \subset \Pi_{2N}(\mathbb{S}^{d-1})$.

The sphere $\mathbb{S}^{d-1}$ is intimately connected to the special orthogonal group $SO(d) = \{ g\in \mathbb{R}^{d\times d}: g^\top = g^{-1}, \; \det g =1  \}$, since $\mathbb{S}^{d-1}$ can be identified with the homogeneous space $SO(d)/SO(d-1)$. Here, we identify $SO(d-1)$ with the subgroup of all elements $g\in SO(d)$ satisfying $g e^d = e^d$. Analogously, when viewed as a subset of $SO(d)$, we consider $SO(m)$, $2 \leq m \leq d-1$, to be equal to
\begin{equation*}
\{g \in SO(d) \mid g e^j = e^j  \text{ for } j = m+1, ..., d \}.
\end{equation*}
By $\mu_d$ we denote the invariant Haar measure on $SO(d)$ with the normalization
\begin{equation*}
\int_{SO(d)} \mathrm{d}\mu_d = 1.
\end{equation*}
Each element $g\in SO(d)$ can be written as $g = g_\eta h$ where $h \in SO(d-1)$ and $g_\eta \in SO(d)$ with $g_\eta e^{d} = \eta \in \mathbb{S}^{d-1}$. Of course, this representation is not unique. However, it holds that
\begin{equation}\label{eq2}
\int_{SO(d)} f(g) \, \mathrm{d}\mu_d(g) = \int_{\mathbb{S}^{d-1}} \int_{SO(d-1)} f(g_\eta h)  \, \mathrm{d}\mu_{d-1}(h) \, \mathrm{d}\omega_{d-1}(\eta)
\end{equation}
for any $f \in L^1(SO(d))$. In particular, the integral over $SO(d-1)$ on the right hand sight does not depend on the choice of $g_\eta$ due to the invariance of the Haar measure. In the following, $g_\eta$ will always denote an arbitrary element in $SO(d)$ satisfying $g_\eta e^d = \eta$. More generally, if $\hat{\eta} \in \mathbb{S}^{m}$, $1 \leq m \leq d-1$, then $g_{\hat{\eta}}$ will denote some rotation matrix in $ SO(m+1)\subset SO(d)$ such that $g_{\hat{\eta}}e^{m+1} = (\hat{\eta}, 0) \in \mathbb{R}^{m+1} \times \mathbb{R}^{d-m-1}$.

We now consider the unitary group representation
\begin{equation*}
g \mapsto T^d(g), \quad T^d(g)f(\eta) = f(g^{-1}\eta),
\end{equation*}
of $SO(d)$ on $L^2(\mathbb{S}^{d-1})$. Its subrepresentations on the spaces $\mathcal{H}_n^d$ are irreducible and thus, by the Peter-Weyl theorem, the matrix coefficients
\begin{equation*}
t_{k, m}^{d,n}\colon SO(d) \rightarrow \mathbb{C}, \quad t_{k, m}^{d,n}(g) = \langle T^d(g) Y_m^{d,n}, Y_k^{d,n} \rangle_{\mathbb{S}^{d-1}},
\end{equation*}
form an orthogonal system of $L^2(SO(d))$ with respect to the inner product
\begin{equation*}
\langle f, g \rangle_{SO(d)} = \int_{ SO(d)} f \, \overline{g} \, \mathrm{d}\mu_d.
\end{equation*}
Additionally, it holds that
\begin{equation}\label{eq23}
\int_{SO(d)} \lvert t_{k, m}^{d,n} \rvert^2 \, \mathrm{d}\mu_d = \frac{1}{\dim \mathcal{H}_n^{d}}.
\end{equation}
In the following, we will frequently make use of the relations
\begin{equation}\label{eq6}
t_{k, \ell}^{d, n}(g \tilde{g}) = \sum_{m \in \mathcal{I}_n^d} t_{k, m}^{d, n}(g) \, t_{m, \ell}^{d, n}(\tilde{g}), \quad g, \tilde{g} \in SO(d),
\end{equation}
and
\begin{equation}\label{eq7}
\sqrt{\dim \mathcal{H}_n^d}\, t_{k, 0}^{d, n}(g_\eta) =  \overline{Y_k^{d, n}(\eta)}, \quad \eta \in \mathbb{S}^{d-1}.
\end{equation}
Also, the fact that $\overline{Y_{k}^{d, n}} = Y_{k^-}^{d, n}$, where $(k_1, ..., k_{d-3}, k_{d-2})^-$ is defined as $(k_1, ...,k_{d-3} -k_{d-2})$, immediately implies
\begin{align*}
\overline{t_{k, \ell}^{d, n}(g)} = \langle \overline{T^d(g)Y_\ell^{d, n}}, \overline{Y_k^{d, n}} \rangle_{\mathbb{S}^{d-1}} = \langle T^d(g)Y_{\ell^-}^{d, n}, Y_{k^{-}}^{d, n} \rangle_{\mathbb{S}^{d-1}} = t_{k^-, \ell^-}^{d, n}(g).
\end{align*}
Another important observation is given by the following lemma.
\begin{lemma}\label{lemma2}
For $d \geq 4$ and $h \in SO(d-1)$ it holds that
\begin{equation*}
t_{k, \ell}^{d, n}(h) = \delta_{k_1, \ell_1} \, t_{(k_2, ..., k_{d-2}), (\ell_2, ..., \ell_{d-2})}^{d-1, k_1}(h).
\end{equation*}
If $d=3$, then
\begin{equation}\label{eq22}
t_{k, \ell}^{3, n}(h(\gamma)) = \delta_{k, l} \, \mathrm{e}^{\mathrm{ik\gamma}}, \quad \gamma \in [0, 2\pi),
\end{equation}
where $h(\gamma)\in \mathbb{R}^{3\times 3}$ is a positive rotation by $\gamma$ in the $(x_1, x_2)$-plane.
\end{lemma}
\begin{proof}
Equation \eqref{eq22} holds, since
\begin{equation*}
T^3(h(\gamma))Y_\ell^{3, n} = \mathrm{e}^{\mathrm{i}\ell \gamma}\, Y_\ell^{3, n}.
\end{equation*}
For  $d \geq 4$ and $k=(k_1, ..., k_{d-2}) \in  \mathcal{I}_{n}^d$ it follows directly from the definition \eqref{y_n^k def} that
\begin{align*}
& Y_k^{d,n}(\theta_1, ..., \theta_{d-1}) =  \frac{A_k^n}{A_{(k_2, ..., k_{d-2})}^{k_1}} C_{n-\lvert k_1 \rvert}^{\frac{d-2}{2}+\lvert k_1 \rvert}(\cos \theta_{d-1}) \\
& \qquad \qquad \qquad \qquad \qquad \qquad \qquad \times \sin^{\lvert k_1 \rvert}(\theta_{d-1})\,  Y_{(k_2, ..., k_{d-2})}^{d-1, \lvert k_1 \rvert }(\theta_1, ..., \theta_{d-2}).
\end{align*}
Consequently, for $h \in SO(d-1)$ we have
\begin{align*}
&T^{d}(h)Y_k^{d,n}(\theta_1, ..., \theta_{d-1}) =  \frac{A_k^n}{A_{(k_2, ..., k_{d-2})}^{\lvert k_1 \rvert }} C_{n-\lvert k_1\rvert }^{\frac{d-2}{2}+\lvert k_1 \rvert}(\cos \theta_{d-1}) \\
& \qquad \qquad \qquad \qquad  \times \sin^{\lvert k_1 \rvert }(\theta_{d-1})\,  T^{d-1}(h)Y_{(k_2, ..., k_{d-2})}^{d-1,\lvert k_1\rvert }(\theta_1, ..., \theta_{d-2}).
\end{align*}
Now, it is easy to see that
\begin{equation*}
t_{k, \ell}^{d, n}(h) = \langle T^d(h)Y_\ell^{d, n}, Y_k^{d, n}\rangle_{\mathbb{S}^{d-1}} =c\, \delta_{k_1, \ell_1} \, t_{(k_2, ..., k_{d-2}), (\ell_2, ..., \ell_{d-2})}^{d-1, k_1}(h)
\end{equation*}
for some constant $c$. Furthermore, by choosing $h$ to be the identity, we deduce that $c$ must be equal to $1$.
\end{proof}

On the rotation groups, we consider the subspaces 
\begin{equation}\label{eq112}
\mathcal{M}_N^1(SO(d)) = \vspan\{ t_{k, m}^{d,n} : n\leq N, \; k, m \in \mathcal{I}_n^d \}, \qquad d\geq 3,
\end{equation}
where the upper index $1$ in \eqref{eq112} refers to the fact that the matrix functions $t_{k, m}^{d, n}$ correspond to class $1$ representations (see \citep{bib32} for more details). Furthermore, we define
\begin{equation*}
\mathcal{M}_N^1(SO(2)) = \Pi_N(SO(2))= \vspan\{\mathrm{e}^{\mathrm{i}k \cdot} :k=-N, ..., N \}.
\end{equation*}
\hyperref[lemma2]{Lemma~\ref*{lemma2}} implies, in particular, that
\begin{equation*}
    f\vert_{SO(m)} \in \mathcal{M}_N^1(SO(m)) \qquad \text{for each }f \in \mathcal{M}_N^1(SO(d)), \; m \in \{2, \dots, d-1\}.
\end{equation*}
For $d \geq 3$, the group $SO(d)$ can be regarded as a subset of $\mathbb{R}^{d^2}$. In that sense, we define $\Pi_N(SO(d))$ as the space of all restrictions of algebraic polynomials of degree $N$ in $d^2$ variables to $SO(d)$. As easily seen from the definition of the matrix functions, it holds that
\begin{equation*}
\mathcal{M}_N^1(SO(d)) \subset \Pi_N(SO(d)).
\end{equation*}
Hence, we obtain a first simple product property
\begin{equation*}
f_1 f_2 \in \Pi_{2N}(SO(d)) \quad \text{if } f_1, f_2 \in \mathcal{M}_N^1(SO(d)).
\end{equation*}
Another result in this direction is given by the following lemma, which establishes a connection between polynomials on the sphere and the rotation group. This property will play an important role in the construction of our wavelet frames.
\begin{lemma}\label{lemma4}
Let $f_1 \in \mathcal{M}_{N_1}^1(SO(d))$, $f_2 \in \mathcal{M}_{N_2}^1(SO(d))$. Then the function
\begin{equation*}
\eta \mapsto \int_{SO(d-1)} f_1(g_\eta h)f_2(g_\eta h) \, \mathrm{d}\mu_{d-1}(h)
\end{equation*}
is contained in $\Pi_{N_1+N_2}(\mathbb{S}^{d-1})$.
\end{lemma}
\begin{proof}
For $d =3$ the result is well known. Indeed, in this case the matrix functions $t_{l, k}^{3, n}$, often called Wigner $D$-functions, satisfy the product property
\begin{equation*}
    t_{\ell, k}^{3, n} \,  t_{\ell', k'}^{3, n'} \in \mathcal{M}_{n+n'}^1(SO(3)).
\end{equation*}
More details in this direction can be found in \cite{bib20}.

Let $d\geq 4$. Clearly, it suffices to consider the case where $f_1 = t_{\ell, k}^{d, n}$ and $f_2 = \overline{t_{\ell', k'}^{d, n'}}$ with $n \leq N_1$, $n'\leq N_2$. A straightforward computation using \eqref{eq6} and \hyperref[lemma2]{Lemma~\ref*{lemma2}} yields
\begin{align*}
\int_{SO(d-1)} f_1(g_\eta h)f_2(g_\eta h) \, \mathrm{d}\mu_{d-1}(h)   =\frac{\delta_{k, k'}}{\dim \mathcal{H}_{k_1}^{d-1}} \sum_{m \in \mathcal{I}_{k_1}^{d-1}} t_{\ell, (k_1, m)}^{d, n}(g_\eta) \, \overline{t_{\ell', (k_1, m)}^{d, n'}(g_\eta)}.
\end{align*}
Now, write $\eta = (\eta_1, ..., \eta_d)$ and
\begin{equation*}
g_\eta = \begin{pmatrix}
g_{1,1} & g_{1, 2} & \cdots & g_{1, d-1} & \eta_1\\
g_{2,1} & g_{2, 2} & \cdots & g_{2, d-1} & \eta_2 \\
\vdots & \vdots &  & \vdots & \vdots \\
g_{d,1} & g_{d, 2} & \cdots & g_{d, d-1} & \eta_d
\end{pmatrix}, \quad g_\eta' = \begin{pmatrix}
g_{1,1} & g_{1, 2} & \cdots & g_{1, d-1}\\
g_{2,1} & g_{2, 2} & \cdots & g_{2, d-1}\\
\vdots & \vdots &  & \vdots \\
g_{d,1} & g_{d, 2} & \cdots & g_{d, d-1}
\end{pmatrix}.
\end{equation*}
As already discussed in the proof of \hyperref[lemma2]{Lemma~\ref*{lemma2}}, for $x = (x_1, ..., x_d)^\top \in \mathbb{S}^{d-1}$ we have 
\begin{equation*}
Y_{(k_1, m)}^{d, n}(x) = c_{n, k_1} C_{n-k_1}^{\frac{d-2}{2}+k_1}(x_d) (1-x_d^2)^{k_1/2} \,  Y_m^{d-1, k_1}\!\bigg( \frac{(x_1, ..., x_{d-1})^\top}{\sqrt{1-x_d^2}} \bigg),
\end{equation*}
with $c_{n, k_1} = A_{(k_1, m)}^n/A_{m}^{k_1}$ being independent of $m$, 
and therefore
\begin{align*}
&T^d(g_\eta)Y_{(k_1, m)}^{d, n}(x) = Y_{(k_1, m)}^{d, n}(g_\eta^\top x) \\
&\qquad =  c_{n, k_1} C_{n-k_1}^{\frac{d-2}{2}+k_1}(\langle \eta, x \rangle) (1-\langle \eta, x \rangle^2)^{k_1/2}\,  Y_m^{d-1, k_1}\!\bigg( \frac{(g_\eta')^\top x}{\sqrt{1-\langle \eta, x \rangle^2}} \bigg).
\end{align*}
Since $g_\eta$ is an orthogonal matrix, it is easy to see that
\begin{equation*}
\langle (g_\eta')^\top x, (g_\eta')^\top z\rangle = z^\top (I_d - \eta \eta^\top)x = \langle x, z\rangle-\langle \eta, z \rangle \langle \eta, x\rangle.
\end{equation*}
The addition theorem for spherical harmonics \eqref{addition theorem} then shows that
\begin{align*}
&\sum_{m \in \mathcal{I}_{k_1}^{d-1}}Y_m^{d-1, k_1}\!\bigg( \frac{(g_\eta')^\top x}{\sqrt{1-\langle \eta, x \rangle^2}}\bigg) \overline{Y_m^{d-1, k_1}\!\bigg( \frac{(g_\eta')^\top z}{\sqrt{1-\langle \eta, z \rangle^2}}\bigg)} \\
&\qquad  \qquad \qquad \qquad = \frac{2k_1 + d-3}{d-3} C_{k_1}^{\frac{d-3}{2}}\!\bigg( \frac{\langle x, z\rangle-\langle \eta, z \rangle \langle \eta, x\rangle}{\sqrt{1-\langle \eta, x \rangle^2} \sqrt{1-\langle \eta, z \rangle^2}} \bigg).
\end{align*}
Hence, combining all the above, we deduce that
\begin{align*}
& \int_{SO(d-1)} f_1(g_\eta h)f_2(g_\eta h) \, \mathrm{d}\mu_{d-1}(h) = c \int_{\mathbb{S}^{d-1}}\int_{\mathbb{S}^{d-1}} \overline{Y_\ell^{d, n}(x)}Y_{\ell'}^{d, n'}(z)\\
& \qquad \quad \qquad  \times C_{n-k_1}^{\frac{d-2}{2}+k_1}(\langle \eta, x \rangle)C_{n'-k_1}^{\frac{d-2}{2}+k_1}(\langle \eta, z \rangle) (1-\langle \eta, x \rangle^2)^{k_1/2}(1-\langle \eta, z \rangle^2)^{k_1/2} \\
& \qquad \quad \qquad \times C_{k_1}^{\frac{d-3}{2}}\!\bigg( \frac{\langle x, z\rangle-\langle \eta, z \rangle \langle \eta, x\rangle}{\sqrt{1-\langle \eta, x \rangle^2} \sqrt{1-\langle \eta, z \rangle^2}} \bigg)\, \mathrm{d}\omega_{d-1}(x)\, \mathrm{d}\omega_{d-1}(z),
\end{align*}
for some constant $c$ depending on $n, n', k, k'$. Using the fact that $C_{k_1}^{\frac{d-3}{2}}$ is a one-dimensional polynomial of degree $k_1$, we have
\begin{equation*}
C_{k_1}^{\frac{d-3}{2}}\!\bigg( \frac{\langle x, z\rangle-\langle \eta, z \rangle \langle \eta, x\rangle}{\sqrt{1-\langle \eta, x \rangle^2} \sqrt{1-\langle \eta, z \rangle^2}} \bigg) = \sum_{j=0}^{k_1} c_j \bigg( \frac{\langle x, z\rangle-\langle \eta, z \rangle \langle \eta, x\rangle}{\sqrt{1-\langle \eta, x \rangle^2} \sqrt{1-\langle \eta, z \rangle^2}} \bigg)^j
\end{equation*}
for certain coefficients $c_j$. Furthermore, 
\begin{equation*}
(\langle x, z\rangle-\langle \eta, z \rangle \langle \eta, x\rangle)^j = \sum_{i=0}^j(-1)^i \binom{j}{i}\langle x, z\rangle^{j-i} \langle \eta, z\rangle^i \langle \eta, x \rangle^i
\end{equation*}
as well as
\begin{equation*}
\langle x, z\rangle^{j-i} = \sum_{\vert \alpha \rvert = j-i} \tilde{c}_\alpha \,x^\alpha z^\alpha,
\end{equation*}
where the above sum ranges over all $\alpha = (\alpha_1, ..., \alpha_d)\in \mathbb{N}_0^d$ with $\lvert \alpha \rvert = \alpha_1 + ...+ \alpha_d=j-i$. It follows that 
\begin{align*}
\int_{SO(d-1)} f_1(g_\eta h)f_2(g_\eta h) \, \mathrm{d}\mu_{d-1}(h)    =\sum_{j=0}^{k_1} \sum_{i=0}^j \sum_{\vert \alpha \rvert = j-i} c_{j, i, \alpha} \, F_{j, i, \beta}(n, \ell, \eta) \overline{F_{j, i, \alpha}(n', \ell', \eta) }
\end{align*}
with
\begin{align}
F_{j, i, \alpha}(n, \ell, \eta ) = \int_{\mathbb{S}^{d-1}} \overline{Y_{\ell}^{d, n}(x)} \,x^\alpha  \,C_{n-k_1}^{\frac{d-2}{2}+k_1}(\langle \eta, x \rangle) \langle \eta, x\rangle^i (1-\langle \eta, x \rangle^2)^{(k_1-j)/2}\, \mathrm{d}\omega_{d-1}(x).
\end{align}
Hence, in order to complete the proof it suffices to show that $F_{j, i, \alpha}(n, \ell, \cdot ) \in \Pi_{n}(\mathbb{S}^{d-1})$. First, we note that the latter is obvious if $k_1-j$ is even. Otherwise, if $k_1-j$ is odd, we will show that $F_{j, i, \alpha}(n, \ell, \cdot ) \equiv 0$. Clearly, the function $x \mapsto \overline{Y_{\ell}^{d, n}(x)} x^\alpha $ is a homogeneous polynomial of degree $n + \lvert \alpha \rvert$ and thus, considered as a function on the sphere, can be expanded into spherical harmonics whose degrees have the same parity as $n+\lvert \alpha \rvert$ (see e.g.\ \citep[Theorem~1.1.3]{bib3}). Similarly, by the symmetry $C_m^{\frac{d-2}{2}}(-t)=(-1)^mC_m^{\frac{d-2}{2}}(t)$, it holds that
\begin{equation*}
C_{n-k_1}^{\frac{d-2}{2}+k_1}(t)t^i (1-t^2)^{(k_1-j)/2} = \sum_{m=0}^\infty c_m \, C_m^{\frac{d-2}{2}}(t)
\end{equation*}
where the Gegenbauer-expansion on the right hand sight converges in the Hilbert space sense and $c_m = 0$ if $m$ has the same parity as $n+i-k_1+1$. Consequently, the corresponding spherical function, obtained by setting $t=\langle \eta, x \rangle$, has an expansion in terms of spherical harmonics whose degree is of the same parity as $n+i-k_1$. Also, by the assumption that $k_1-j$ is odd, $n+i-k_1$ has the same parity as $n+j-i +1$. It therefore follows that $F_{j, i, \alpha}(n, \ell, \cdot ) \equiv 0$.
\end{proof}

We conclude this section, by introducing the subspaces
\begin{equation*}
\mathcal{M}_{N, M}^1(SO(d)) = \vspan\{ t_{k, m}^{d,n} : n\leq N, \; k \in \mathcal{I}_n^d, \; m \in \mathcal{I}_{\min(n, M)}^d \},
\end{equation*}
which will play an important role in the remainder of this article. Note that $\mathcal{M}_{N, M}^1(SO(d))$ $=\mathcal{M}_{N}^1(SO(d))$ if $M\geq N$. While the above definition is convenient when constructing polynomial frames in terms of the specific basis functions $Y_k^{d, n}$ given in \eqref{y_n^k def}, the subspaces $\mathcal{M}_{N, M}^1(SO(d))$ can also be formulated in a coordinate independent way. Namely, we have the following statement.
\begin{proposition}\label{lemma2.3}
Let $f \in \mathcal{M}_{N}^1(SO(d))$. Then $f \in \mathcal{M}_{N, M}^1(SO(d))$ if and only if for each $g \in SO(d)$ the map $h \mapsto f(g h)$ is contained in $\mathcal{M}_{M}^1(SO(d-1))$.
\end{proposition}
\begin{proof}
Since $f \in \mathcal{M}_{N}^1(SO(d))$, we have
\begin{equation*}
f = \sum_{n=0}^N \sum_{\ell, k \in \mathcal{I}_n^d} c_{\ell, k}^n \, t_{\ell, k}^{d, n}
\end{equation*}
for some coefficients $c_{\ell, k}^n \in \mathbb{C}$. Thus, using \eqref{eq6} and \hyperref[lemma2]{Lemma~\ref*{lemma2}}, we obtain
\begin{equation*}
f(g h) = \sum_{n=0}^N \sum_{\ell, k \in \mathcal{I}_n^d} c_{\ell, k}^n \sum_{m \in \mathcal{I}_{k_1}^{d-1}}t_{\ell, (k_1, m) }^{d, n}(g) \, t_{m, (k_2, ..., k_{d-2})}^{d-1, k_1}(h).
\end{equation*}
Now, the statement is obvious.
\end{proof}

\section{Constructing directional frames}\label{sec3}
In this section, we construct a large class of polynomial frames for $L^2(\mathbb{S}^{d-1})$. We start by defining an initial sequence of polynomial wavelets 
\begin{equation*}
\Psi_{\scriptscriptstyle K}^j\in \Pi_{2^j}(\mathbb{S}^{d-1}), \quad j \in \mathbb{N}_0,
\end{equation*}
where $K$ is a fixed directionality parameter. For large scales $j$, these wavelets are well localized in space, as we will discuss in more detail in \autoref{sec4}. The frames we construct in this section consist of functions of the form $T^d(g_{j, \ell})\Psi_{\scriptscriptstyle K}^j$, i.e., of rotated versions of the initial wavelets. The points $g_{j, \ell}\in SO(d)$ correspond to quadrature formulas for suitable polynomial subspaces of $L^2(SO(d))$.

Let $\phi \in C^q([0, \infty))$, $q \in \mathbb{N}_0$, be a non-increasing function with $\phi(t)=1$ for $t\in [0, 1/2]$ and $\phi(t)=0$ for $t\geq 1$. We define $\kappa \in C^q([0, \infty))$ by
\begin{equation*}
\kappa(t)=\sqrt{\phi^2(t/2) - \phi^2(t)}.
\end{equation*}
It follows that $\supp(\kappa) \subset [1/2, 2]$ and
\begin{equation}\label{eqkappa}
\phi^2(t) + \kappa^2(t) = \phi^2(t/2).
\end{equation}
Now, for $j \in \mathbb{N}_0$, we consider the scaling function
\begin{equation}\label{eq25}
\Phi^j = \sum_{n=0}^{\infty}\sum_{k  \in \mathcal{I}_n^d} \phi^2\!\left(\frac{n}{2^j} \right) \overline{Y_k^{d, n}(e^d)} \, Y_k^{d, n}.
\end{equation}
Note that the addition theorem \eqref{addition theorem} yields an alternative representation in terms of Gegenbauer polynomials. Namely,
\begin{equation*}
\Phi^j (\eta) = \sum_{n=0}^\infty \frac{2n+d-2}{d-2} \, \phi^2\!\left(\frac{n}{2^j} \right) C_n^{\frac{d-2}{2}}(\langle \eta, e^d\rangle).
\end{equation*}
The isotropic scaling function $\Phi^j$ is later used to capture the low-frequency content of the signals under consideration. For the analysis at finer scales we introduce the wavelets
\begin{equation}\label{waveletdef}
\Psi_{\scriptscriptstyle K}^{j} = \sum_{n=0}^{\infty}\sum_{k\in  \mathcal{I}_{n}^d} \sqrt{\dim \mathcal{H}_n^d} \, \kappa\!\left(\frac{n}{2^{j-1}} \right) \zeta_{ K, k}^{d,n} \,  Y_k^{d,n}, \quad j \in \mathbb{N}.
\end{equation}
Here, the parameter $K\in \mathbb{N}_0$ is a fixed predefined value, which puts restrictions onto the directionality components $\zeta_{ K, k}^{d,n}$. Namely, we require that $\zeta_{ K, k}^{d,n}$ is independent of $n$ if $n \geq K$,  in which case we will simply write $\zeta_{K, k}^d$ instead of $\zeta_{ K, k}^{d,n}$, as well as $\zeta_{ K, k}^{d, n} = 0$ whenever $k=(k_1, k_2,..., k_{d-2})$ with $k_1 >K$. Additionally, we assume that
\begin{equation}\label{eq1}
\sum_{k\in  \mathcal{I}_{n}^d} \lvert \zeta_{ K, k}^{d,n}\rvert^2 = 1, \quad n \in \mathbb{N}.
\end{equation}
In the special case $K=0$ we obtain isotropic wavelets of the form
\begin{equation*}
\Psi_{\scriptscriptstyle 0}^j (\eta) = \sum_{n=0}^\infty \frac{2n+d-2}{d-2} \kappa\!\left(\frac{n}{2^{j-1}} \right) C_n^{\frac{d-2}{2}}(\langle \eta, e^d \rangle),
\end{equation*}
which have been extensively investigated in the literature \citep{bib3, bib11, bib14, bib15, bib16, bib30, bib35, bib36, bib37}. For $K>0$, the resulting wavelets can exhibit more complex directional structures, which makes them suitable for direction sensitive analyses of local features. Indeed, for $d=3$, where the above construction yields the directional wavelets studied in \citep{bib21, bib10, bib9}, we have previously demonstrated their ability to identify edges and higher order singularities in terms of their position and orientation \citep{bib26, bib31}.

Apart from the mentioned restrictions, the values $\zeta_{ K, k}^{d, n}$ can be chosen freely. However, in \autoref{sec6} we will propose specific choices that maximize the directional sensitivity of the corresponding wavelets. It is also worth mentioning that condition \eqref{eq1} alone is sufficient to ensure that our construction yields tight frames. As we will show later, the remaining assumptions on the directionality component have two main implications. On the one hand, the resulting wavelets are well localized in space. On the other hand, the amount of samples needed to exactly reconstruct band-limited signals can be greatly reduced. Indeed, while the discretized anisotropic wavelet transform in general requires the sampling of polynomials on $SO(d)$, in our setting the amount of quadrature points needed is comparable to the isotropic case where only polynomials on $\mathbb{S}^{d-1}$ occur. Closely related to this is the fact that our wavelets are steerable, a property which will be discussed at the end of \autoref{sec4}.

For any integer $j \in \mathbb{N}$, let $\eta_{j, \ell} \in \mathbb{S}^{d-1}$, $\ell = 1, ..., s_j$, be quadrature points with non-negative weights $\omega_{j, \ell}$ such that
\begin{equation}\label{eq15}
\sum_{\ell =1}^{s_j} \omega_{j, \ell} \, f(\eta_{j, \ell}) =\int_{\mathbb{S}^{d-1}} f \, \mathrm{d}\omega_{d-1}  \quad \text{for all } f \in \Pi_{2^{j+1}}(\mathbb{S}^{d-1}).
\end{equation}
Furthermore, let $h_{j,m} \in SO(d-1)$ and $\mu_{j, m}>0$ be given for $m = 1, ..., r_j$ with
\begin{equation}\label{eq14}
\sum_{m =1}^{r_j} \mu_{j,m} \, f(h_{j, m}) =\int_{SO(d-1)} f \, \mathrm{d}\mu_{d-1} \quad \text{for all } f \in \Pi_{2\min(K,2^j)}(SO(d-1)).
\end{equation}
Note that $h_{j, m}$ and $\mu_{j, m}$ can be chosen independent of the scale $j$ for $2^j\geq K$, and henceforth we will assume that they are. It will be convenient to define $r_0 = s_0=1$. Combining \eqref{eq15} and \eqref{eq14}, we obtain exact quadrature formulas for the products of polynomials in $\mathcal{M}_{2^j, K}^1(SO(d))$ introduced in \autoref{sec2}.
\begin{lemma}\label{lemma6}
Let $j\in \mathbb{N}$ and $f_1, f_2 \in \mathcal{M}_{2^j, K}(SO(d))$. Then we have
\begin{equation*}
\sum_{\ell=1}^{s_{j}}\sum_{m=1}^{r_j} \omega_{j, \ell} \, \mu_{j, m} \, f_1(g_{\eta_{j, \ell}} h_{j, m})  \, f_2(g_{\eta_{j, \ell}} h_{j, m}) = \int_{SO(d)} f_1 f_2\, \mathrm{d}\mu_{d}.
\end{equation*}
\end{lemma}
\begin{proof}
The statement follows easily by utilizing \eqref{eq2} as well as \hyperref[lemma4]{Lemma~\ref*{lemma4}} and \hyperref[lemma2.3]{Proposition~\ref*{lemma2.3}}. Indeed, \eqref{eq14} can be replaced by the weaker condition
\begin{equation*}
\sum_{m =1}^{r_j} \mu_{j,m} \, f_1(h_{j, m}) f_2(h_{j, m}) =\int_{SO(d-1)} f_1 f_2 \, \mathrm{d}\mu_{d-1} \quad \text{if } f_1, f_2 \in \mathcal{M}_{\min(K,2^j)}^1(SO(d-1)).
\end{equation*}
\end{proof}

By the above construction of $\Psi_{\scriptscriptstyle K}^{j}$, the function $g\mapsto T^d(g)\Psi_{\scriptscriptstyle K}^{j}(\eta)$ is an element of $\mathcal{M}_{2^j, K}^1(SO(d))$ for any $\eta \in \mathbb{S}^{d-1}$. Indeed, we have
\begin{equation}\label{eq102}
T^d(g)\Psi_{\scriptscriptstyle K}^{j} = \sum_{n=0}^\infty \sum_{k, k'\in  \mathcal{I}_{n}^d} \sqrt{\dim \mathcal{H}_n^d} \, \kappa\!\left(\frac{n}{2^{j-1}} \right) \zeta_{K, k'}^{d,n}\, t_{k, k'}^{d, n}(g)  Y_k^{d,n}.
\end{equation}
Consequently, the mapping $g \mapsto \langle f, T^d(g)\Psi_{\scriptscriptstyle K}^{j} \rangle_{\mathbb{S}^{d-1}}$ is also contained in the polynomial subspace $\mathcal{M}_{2^j, K}^1(SO(d))$, whenever $f \in L^1(\mathbb{S}^{d-1})$. Hence, \hyperref[lemma6]{Lemma~\ref*{lemma6}} can be applied to the functions
\begin{equation*}
g \mapsto \langle f, T^d(g)\Psi_{\scriptscriptstyle K}^{j} \rangle_{\mathbb{S}^{d-1}} \, T^d(g)\Psi_{\scriptscriptstyle K}^{j}(\eta)
\end{equation*}
and
\begin{equation*}
g \mapsto \lvert\langle f, T^d(g)\Psi_{\scriptscriptstyle K}^{j} \rangle_{\mathbb{S}^{d-1}}\rvert^2
\end{equation*}
on $SO(d)$.

We can now define the fully discretized directional wavelets
\begin{equation*}
\Psi_{\scriptscriptstyle K}^{ 0, 1, 1} \equiv  1, \quad \Psi_{\scriptscriptstyle K}^{j, \ell, m} = \sqrt{\omega_{j, \ell}\, \mu_{j,m}} \, T^d(g_{\eta_{j, \ell}} h_{j,m})\Psi_{\scriptscriptstyle K}^{j}, \quad j \geq 1,
\end{equation*}
which form a tight frame for $L^2(\mathbb{S}^{d-1})$ as the following proposition states.
\begin{proposition}
Let $f \in L^2(\mathbb{S}^{d-1})$. Then
\begin{equation*}
\sum_{j=0}^\infty \sum_{\ell=1}^{s_{j}} \sum_{m=1}^{r_j} \lvert \langle f,\Psi_{\scriptscriptstyle K}^{j, \ell, m}\rangle_{\mathbb{S}^{d-1}}\rvert^2 = \int_{\mathbb{S}^{d-1}} \lvert f \rvert^2 \, \mathrm{d}\omega_{d-1}.
\end{equation*}
\end{proposition}
\begin{proof}
For $j=0$ we have $\lvert \langle f,\Psi_{\scriptscriptstyle K}^{0, 1, 1}\rangle_{\mathbb{S}^{d-1}}\rvert^2 = \lvert \langle f, Y_0^{d,0}   \rangle_{\mathbb{S}^{d-1}}\rvert^2$. If $j\geq 1$, then
\begin{equation*}
\lvert \langle f,\Psi_{\scriptscriptstyle K}^{j, \ell, m}\rangle_{\mathbb{S}^{d-1}}\rvert^2 = \omega_{j, \ell}\, \mu_{j,m} \lvert \langle f,T^d(g_{\eta_{j, \ell}} h_{j,m})\Psi_{\scriptscriptstyle K}^{j}\rangle_{\mathbb{S}^{d-1}}\rvert^2.
\end{equation*}
Hence, \hyperref[lemma6]{Lemma~\ref*{lemma6}} yields
\begin{equation*}
\sum_{\ell=1}^{s_{j}} \sum_{m=1}^{r_j} \lvert \langle f,\Psi_{\scriptscriptstyle K}^{j, \ell, m}\rangle_{\mathbb{S}^{d-1}}\rvert^2 = \int_{SO(d)} \lvert\langle f, T^d(g)\Psi_{\scriptscriptstyle K}^{j} \rangle_{\mathbb{S}^{d-1}}\rvert^2 \, \mathrm{d}\mu_{d}(g).
\end{equation*}
Using the fact that
\begin{align*}
\langle T^d(g)\Psi_{\scriptscriptstyle K}^{j}, f \rangle_{\mathbb{S}^{d-1}}  =\sum_{n=0}^\infty \sum_{k, k'\in  \mathcal{I}_{n}^d} \sqrt{\dim \mathcal{H}_n^d} \, \kappa\!\left(\frac{n}{2^{j-1}} \right) \zeta_{ K, k'}^{d,n}\, \langle  Y_k^{d,n}, f \rangle_{\mathbb{S}^{d-1}} \, t_{k, k'}^{d, n}(g),
\end{align*}
we deduce from the orthogonality of the matrix coefficients $t_{k, k'}^{d, n}$, as well as from \eqref{eq23} and \eqref{eq1}, the equality
\begin{equation*}
\sum_{\ell=1}^{s_{j}} \sum_{m=1}^{r_j} \lvert \langle f,\Psi_{\scriptscriptstyle K}^{j, \ell, m}\rangle_{\mathbb{S}^{d-1}}\rvert^2 = \sum_{n=0}^\infty \kappa^2\!\left(\frac{n}{2^{j-1}} \right) \lvert \langle f, Y_k^{d, n} \rangle_{\mathbb{S}^{d-1}}\rvert^2.
\end{equation*}
The frame property now follows from the construction of $\kappa$ and Parseval's identity.
\end{proof}

\begin{corollary}
For $f \in L^2(\mathbb{S}^{d-1})$ we have
\begin{equation*}
f = \sum_{j=0}^\infty \sum_{\ell=1}^{s_{j}} \sum_{m=1}^{r_j} \langle f,\Psi_{\scriptscriptstyle K}^{j, \ell, m}\rangle_{\mathbb{S}^{d-1}}\Psi_{\scriptscriptstyle K}^{j, \ell, m},
\end{equation*}
where the series on the right hand sight converges unconditionally with respect to $\| \cdot \|_{L^2(\mathbb{S}^{d-1})}$.
\end{corollary}

To end this section, we want to comment on the quadrature formulas used in our construction. Due to the limited directionality of the wavelets $\Psi_{\scriptscriptstyle K}^j$, the transformation
\begin{equation*}
g \mapsto \langle f, T^d(g)\Psi_{\scriptscriptstyle K}^j \rangle_{\mathbb{S}^{d-1}}
\end{equation*}
lies in $\mathcal{M}_{2^j, K}^1(SO(d))$, which allows us to use products of quadrature rules for $\Pi_{2^{j+1}}(\mathbb{S}^{d-1})$ and $\Pi_{2K}(SO(d-1))$ at each scale $j$ (see \hyperref[lemma6]{Lemma~\ref*{lemma6}}). Here, we can assume that the amount of samples $s_j$ is comparable to the dimension of the corresponding polynomial subspace on the sphere \citep{bib34}. I.e.,
\begin{equation}\label{eq100}
s_j \sim \dim \Pi_{2^{j+1}}(\mathbb{S}^{d-1}) \sim 2^{j(d-1)},
\end{equation}
where we use the notation $A \sim B$ to indicate that there exist constants $c_1, c_2>0$ such that $ c_1 B \leq A \leq c_2 B $. The constants of equivalence in \eqref{eq100} dependent only on the dimension $d$. In particular, the amount of samples needed is comparable to the isotropic case (see \autoref{subsec6} for more details).

\section{Localization and directional sensitivity}\label{sec4}
In this section, we will use the notation
\begin{equation}\label{eq11}
\Psi_{\scriptscriptstyle K}^{\scriptscriptstyle N} = \sum_{n=0}^{\infty}\sum_{k\in  \mathcal{I}_{n}^d} \sqrt{\dim \mathcal{H}_n^d} \, \kappa\!\left( \frac{n}{N} \right) \zeta_{K, k}^{d,n} \,  Y_k^{d,n}, \quad N \in \mathbb{N},
\end{equation}
with the same assumptions on $\kappa$ and $\zeta_{ K, k}^{d,n}$ as discussed earlier. In the isotropic case, where $K=0$, localization bounds for wavelets of the type $\Psi_{\scriptscriptstyle 0}^{\scriptscriptstyle N}$ are well known (see e.g.\ \citep{bib3, bib15, bib16}). However, when it comes to the more complex directional systems, only the two-dimensional setting has been investigated \citep{bib9, bib26, bib31}. The following theorem now extends these previous results by providing a unified localization bound for isotropic and non-isotropic polynomial wavelets on all spheres $\mathbb{S}^{d-1}$, $d\geq 3$. Here, and throughout the remainder of the article, the occurring constants $c, c_0, c_1, ...$ only depend on the dimension $d$ and the indicated parameters. Furthermore, their exact values might change with each appearance.

\begin{theorem}\label{theorem1}
Let $\Psi_{\scriptscriptstyle K}^{\scriptscriptstyle N}$ be as in \eqref{eq11} with $\kappa \in C^{3q-1}([0, \infty))$, $q \in \mathbb{N}$. Then it holds that
\begin{equation}\label{eq101}
\lvert \Psi_{\scriptscriptstyle K}^{\scriptscriptstyle N}(\eta)\rvert \leq \frac{c_{K, q} N^{d-1}}{(1+N \dist(\eta, e^d))^q} \sum_{m=0}^K [N \sin( \dist(\eta, e^d)) ]^{m}, \quad \eta \in \mathbb{S}^{d-1}.
\end{equation}
In particular,
\begin{equation}\label{eq12}
\lvert \Psi_{\scriptscriptstyle K}^{\scriptscriptstyle N}(\eta)\rvert \leq \frac{c_{K, q}N^{d-1}}{(1+N \dist(\eta, e^d))^{q-K}}, \quad \eta \in \mathbb{S}^{d-1}.
\end{equation}
\end{theorem}
\begin{proof}
The following proof uses arguments similar to those in the two-dimensional setting \citep[Proposition~3.1]{bib31}. First, we note that the normalization factor defined in \eqref{eq10} can be written as
\begin{align*}
A_k^n = \sqrt{\frac{(n-\lvert k_1 \rvert)! (2n+d-2)}{\Gamma(n+\lvert k_1\rvert +d-2)}} \tilde{A}_k,
\end{align*}
with $\tilde{A}_k$ being independent of $n$. Thus, by using the well known property
\begin{equation*}
C_{n-\lvert k_1\rvert }^{\frac{d-2}{2}+\lvert k_1 \rvert }(\cos \theta_{d-1}) = \frac{\Gamma ( \frac{d-2}{2} )}{2^{\lvert k_1 \rvert} \Gamma ( \frac{d-2}{2} + \lvert k_1 \rvert )} \frac{\mathrm{d}^{\lvert k_1\rvert}}{\mathrm{d}t^{\lvert k_1\rvert }} C_n^{\frac{d-2}{2}}(\cos \theta_{d-1})
\end{equation*}
of the Gegenbauer polynomials (see e.g.\ \citep[p.~460]{bib32}), we obtain
\begin{align*}
&Y_k^n(\theta_1, ..., \theta_{d-1}) = \sqrt{\frac{(n-\lvert k_1 \rvert)! (2n+d-2)}{\Gamma(n+\lvert k_1\rvert +d-2)}}  \frac{\mathrm{d}^{\lvert k_1\rvert}}{\mathrm{d}t^{\lvert k_1\rvert }} C_n^{\frac{d-2}{2}}(\cos \theta_{d-1})\\
& \qquad \qquad \qquad \qquad \times \sin^{\lvert k_1 \rvert }(\theta_{d-1})\,  h_k(\theta_1, ..., \theta_{d-2}),
\end{align*}
where
\begin{align*}
&h_k(\theta_1, ..., \theta_{d-2})=  \frac{\tilde{A}_k  \Gamma(\frac{d-2}{2} )}{2^{\lvert k_1\rvert } \Gamma ( \frac{d-2}{2} +\lvert k_1\rvert )} \\
& \quad \qquad \qquad  \times \prod_{j=1}^{d-3} C_{ k_j  - \lvert k_{j+1}\rvert}^{\frac{d-j-2}{2}+\lvert k_{j+1}\rvert}(\cos \theta_{d-j-1}) \, \sin^{\lvert k_{j+1}\rvert}(\theta_{d-j-1}) \, \mathrm{e}^{\mathrm{i}k_{d-2}\theta_1}.
\end{align*}
We will assume that $N \geq 2 K$. Then, all non-vanishing summands in \eqref{eq11} satisfy $n \geq K$. This means that all occurring directionality components are independent of $n$ and we can switch the order of summation. Namely,
\begin{align*}
& \Psi_{\scriptscriptstyle K}^{\scriptscriptstyle N}(\theta_1, ..., \theta_{d-1}) = \sum_{k \in \mathcal{I}_K^d} h_{k}(\theta_1, ..., \theta_{d-2}) \,  \zeta_{ K, k}^{d} \, \sin^{\lvert k_1\rvert }(\theta_{d-1})\\
& \qquad  \times \sum_{n=0}^\infty \sqrt{\frac{ \dim \mathcal{H}_n^d \, (n-\lvert k_1\rvert )! (2n+d-2)}{\Gamma(n+\lvert k_1\rvert +d-2)}} \,  \kappa\! \left( \frac{n}{N} \right) \frac{\mathrm{d}^{\lvert k_1\rvert }}{\mathrm{d}t^{\lvert k_1\rvert }}C_n^{\frac{d-2}{2}}(\cos \theta_{d-1}).
\end{align*}
It is not hard to verify that for any fixed $\lvert k_1 \rvert \leq K$ and $p \in \mathbb{N}_0$, there exist constants $c_0,c_1, ..., c_{p-1}$ such that
\begin{align*}
\sqrt{\frac{ \dim \mathcal{H}_n^d \, (n-\lvert k_1\rvert )! (2n+d-2)}{\Gamma(n+\lvert k_1\rvert +d-2)}}  = (2n+d-2)n^{-\lvert k_1 \rvert }\left(\sum_{j=0}^{p-1} c_j n^{-j} +\mathcal{O}(n^{-p}) \right).
\end{align*}
Consequently, we obtain 
\begin{align*}
& \sum_{n=0}^\infty \sqrt{\frac{ \dim \mathcal{H}_n^d \, (n-\lvert k_1\rvert )! (2n+d-2)}{\Gamma(n+\lvert k_1\rvert +d-2)}} \,  \kappa\! \left( \frac{n}{N} \right) \frac{\mathrm{d}^{\lvert k_1\rvert }}{\mathrm{d}t^{\lvert k_1\rvert }}C_n^{\frac{d-2}{2}}(\cos \theta_{d-1})\\
&\quad \qquad \qquad = \sum_{j=0}^{p-1}c_j \sum_{n=0}^\infty(2n+d-2)\,\kappa\! \left( \frac{n}{N} \right)n^{-\lvert k_1\rvert-j}\frac{\mathrm{d}^{\lvert k_1\rvert }}{\mathrm{d}t^{\lvert k_1\rvert }}C_n^{\frac{d-2}{2}}(\cos \theta_{d-1})\\
& \quad \qquad \quad \qquad   + R(N, \lvert k_1\rvert , \theta_{d-1})
\end{align*}
with
\begin{equation*}
\lvert R(N, \lvert k_1\rvert , \theta_{d-1}) \rvert \leq c \sum_{n=0}^\infty\kappa\! \left( \frac{n}{N} \right)n^{-\lvert k_1\rvert-p+1} \sup_{s\in [-1,1]}\left\lvert \frac{\mathrm{d}^{\lvert k_1\rvert }}{\mathrm{d}t^{\lvert k_1\rvert}}C_n^{\frac{d-2}{2}}(s) \right\rvert.
\end{equation*}
The classical Markov inequality for algebraic polynomials yields
\begin{equation*}
\left\lvert \frac{\mathrm{d}^{\lvert k_1\rvert }}{\mathrm{d}t^{\lvert k_1\rvert }}C_n^{\frac{d-2}{2}}(s) \right\rvert \leq n^{2\lvert k_1\rvert} \sup_{s\in [-1, 1] } \left \lvert C_n^{\frac{d-2}{2}}(s) \right\rvert \leq c\, n^{2\lvert k_1 \rvert + d-3}
\end{equation*}
and therefore $\lvert R(N, \lvert k_1\rvert , \theta_{d-1}) \rvert \leq c N^{\lvert k_1\rvert +d-p-1}$. For our purposes it suffices to choose $p=\lvert k_1\rvert+d+q-1$, as this yields $\lvert R(N, \lvert k_1\rvert , \theta_{d-1}) \rvert \leq c N^{-q}$. Furthermore,
\begin{align*}
&\sum_{n=0}^\infty(2n+d-2)\, \kappa\! \left( \frac{n}{N} \right)n^{-\lvert k_1\rvert -j}\frac{\mathrm{d}^{\lvert k_1\rvert }}{\mathrm{d}t^{\lvert k_1\rvert }}C_n^{\frac{d-2}{2}}(\cos \theta_{d-1})\\
&\qquad \qquad = N^{-\lvert k_1\rvert -j}\sum_{n=0}^\infty (2n+d-2)\, \kappa_j\! \left( \frac{n}{N} \right)\frac{\mathrm{d}^{\lvert k_1\rvert }}{\mathrm{d}t^{\lvert k_1\rvert }}C_n^{\frac{d-2}{2}}(\cos \theta_{d-1}),
\end{align*}
where $\kappa_j(t) = \kappa(t)t^{-\lvert k_1\rvert -j}$. Finally, by \citep[Theorem~2.6.7.]{bib3}, it holds that
\begin{align*}
\left\lvert\sum_{n=0}^\infty (2n+d-2)\, \kappa_j\!\left( \frac{n}{N} \right)\frac{\mathrm{d}^{\lvert k_1\rvert }}{\mathrm{d}t^{\lvert k_1\rvert }}C_n^{\frac{d-2}{2}}(\cos \theta_{d-1})\right\rvert \leq \frac{c_{ q}N^{d-1 + 2\lvert k_1\rvert } }{(1+N \theta_{d-1})^q}.
\end{align*}
This completes the proof of \eqref{eq101}. Now, in order to verify \eqref{eq12} it suffices to show that
\begin{equation*}
[N \sin( \dist(\eta, e^d)) ]^{m} \leq c (1+N \dist(\eta, e^d))^{K}, \quad m=0, 1, ..., K.
\end{equation*}
This, however, is obvious.
\end{proof}
The localization bound in \autoref{theorem1} can be utilized to estimate the $L^p$-norms of $\Psi_{\scriptscriptstyle K}^{\scriptscriptstyle N}$.

\begin{corollary}\label{corollary1}
Let $0<p<\infty$ and let $\Psi_{\scriptscriptstyle K}^{\scriptscriptstyle N}$ be the wavelet defined in \eqref{eq11} with $\kappa \in C^{3q-1}([0, \infty))$, $q\geq K+d/p$. Then,
\begin{equation*}
\|\Psi_{\scriptscriptstyle K}^{\scriptscriptstyle N} \|_{L^p(\mathbb{S}^{d-1})} \sim N^{(d-1)(1-1/p)}.
\end{equation*}
Additionally, for $q\geq K + d/2$, there exists some $c>0$ such that
\begin{equation*}
\|\Psi_{\scriptscriptstyle K}^{\scriptscriptstyle N} \|_{L^\infty(\mathbb{S}^{d-1})} \sim \sup_{\eta \in C(e^d, c/N)} \lvert \Psi_{\scriptscriptstyle K}^{\scriptscriptstyle N}(\eta) \rvert \sim N^{d-1}.
\end{equation*}
Here, the constants of equivalence only depend on $\kappa, K, q, p$ and $d$.
\end{corollary}

\begin{proof}
We follow closely the proof for the isotropic case in \citep{bib30}. For $p=2$, the statement follows easily from \eqref{eq11} and Parseval's identity. In the general case $p\in (0, \infty)$,  we can use the localization bound \eqref{eq12} to obtain
\begin{equation*}
\|\Psi_{\scriptscriptstyle K}^{\scriptscriptstyle N} \|_{L^p(\mathbb{S}^{d-1})}^p \leq c N^{(d-1)p} \int_0^\pi \frac{\theta^{d-2}}{(1+N  \theta)^{(q-K)p}} \mathrm{d}\theta.
\end{equation*}
Now, repeated integration by parts yields 
\begin{equation*}
\|\Psi_{\scriptscriptstyle K}^{\scriptscriptstyle N} \|_{L^p(\mathbb{S}^{d-1})} \leq cN^{(d-1)(1-1/p)}.
\end{equation*}
Also, the estimate $\|\Psi_{\scriptscriptstyle K}^{\scriptscriptstyle N} \|_{L^{\infty}(\mathbb{S}^{d-1})} \leq cN^{(d-1)}$ is a direct consequence of \eqref{eq12} as well. The latter can be used to obtain the corresponding lower bound in the case $0<p<2$, since
\begin{equation*}
\|\Psi_{\scriptscriptstyle K}^{\scriptscriptstyle N} \|_{L^2(\mathbb{S}^{d-1})}^2 \leq \|\Psi_{\scriptscriptstyle K}^{\scriptscriptstyle N} \|_{L^{\infty}(\mathbb{S}^{d-1})}^{2-p} \|\Psi_{\scriptscriptstyle K}^{\scriptscriptstyle N} \|_{L^p(\mathbb{S}^{d-1})}^p.
\end{equation*}
For $2<p\leq \infty$, the lower estimate follows from the above and from Hölder's inequality. Precisely, we have
\begin{equation*}
\|\Psi_{\scriptscriptstyle K}^{\scriptscriptstyle N} \|_{L^2(\mathbb{S}^{d-1})}^2 \leq \|\Psi_{\scriptscriptstyle K}^{\scriptscriptstyle N} \|_{L^p(\mathbb{S}^{d-1})} \|\Psi_{\scriptscriptstyle K}^{\scriptscriptstyle N} \|_{L^q(\mathbb{S}^{d-1})}, \quad q = \frac{p}{p-1},
\end{equation*}
for $2<p<\infty$, as well as
\begin{equation*}
\|\Psi_{\scriptscriptstyle K}^{\scriptscriptstyle N} \|_{L^2(\mathbb{S}^{d-1})}^2 \leq \|\Psi_{\scriptscriptstyle K}^{\scriptscriptstyle N} \|_{L^\infty(\mathbb{S}^{d-1})}  \|\Psi_{\scriptscriptstyle K}^{\scriptscriptstyle N} \|_{L^1(\mathbb{S}^{d-1})}.
\end{equation*}
It remains to confirm the existence of constants $c, c_1>0$ such that
\begin{equation*}
\sup_{\eta \in C(e^d, c/N)} \lvert \Psi_{\scriptscriptstyle K}^{\scriptscriptstyle N}(\eta) \rvert \geq c_1 N^{d-1}.
\end{equation*}
Using the localization bound \eqref{eq12}, we first obtain
\begin{align*}
  c_1 N^{d-1}  \leq \|\Psi_{\scriptscriptstyle K}^{\scriptscriptstyle N} \|_{L^2(\mathbb{S}^{d-1})}^2 \leq \sup_{\eta \in C(e^d, r)} \lvert \Psi_{\scriptscriptstyle K}^{\scriptscriptstyle N}(\eta) \rvert^2 \, \omega_{d-1}( C(e^d, r)) + \int_{r}^\pi \frac{c_{K, q}^2N^{2(d-1)} \theta^{d-2}}{(1+N \theta)^{2(q-K)}} \, \mathrm{d}\theta,
\end{align*}
for each $r \in (0, \pi)$. Repeated integration by parts shows that
\begin{equation*}
\int_{r}^\pi \frac{c_{K, q}^2N^{2(d-1)} \theta^{d-2}}{(1+N \theta)^{2(q-K)}} \, \mathrm{d}\theta \leq \frac{c_2 N^{d-1}}{(1+N r)^{2(q-K)-(d-1)}}.
\end{equation*}
Also, $ \omega_{d-1}( C(e^d, r))  \sim r^{d-1}$. Hence, for $r=cN^{-1}$ we have
\begin{equation*}
N^{2(d-1)} \left(c_1 - \frac{c_2}{(1+c)^{2(q-K) - (d-1)}} \right)\leq c \cdot c_3 \sup_{\eta \in C(e^d, r)} \lvert \Psi_{\scriptscriptstyle K}^{\scriptscriptstyle N}(\eta) \rvert^2.
\end{equation*}
Choosing $c$ large enough completes the proof.
\end{proof}

By \autoref{theorem1}, it is clear that $T^d(g)\Psi_{\scriptscriptstyle K}^{\scriptscriptstyle N}$ is localized at the north pole $e^d$ if and only if $g \in SO(d-1)$. It is therefore natural to use the auto-correlation  function
\begin{equation*}
SO(d-1)\rightarrow \mathbb{C}, \qquad h \mapsto \langle T^d(h)\Psi_{\scriptscriptstyle K}^{\scriptscriptstyle N} , \Psi_{\scriptscriptstyle K}^{\scriptscriptstyle N} \rangle_{\mathbb{S}^{d-1}},
\end{equation*}
as a means to measure and judge the directional sensitivity of $\Psi_{\scriptscriptstyle K}^{\scriptscriptstyle N}$. We note, that this method was also previously used in the two-dimensional setting \citep{bib9, bib21}. \hyperref[lemma2]{Lemma~\ref*{lemma2}} yields the following representation of the auto-correlation function.

\begin{proposition}\label{prop3.2}
Let $d\geq 4$ and $h \in SO(d-1)$. Then it holds that
\begin{align*}
& \langle T^d(h)\Psi_{\scriptscriptstyle K}^{\scriptscriptstyle N} , \Psi_{\scriptscriptstyle K}^{\scriptscriptstyle N}  \rangle_{\mathbb{S}^{d-1}} = \sum_{n=0}^\infty \dim \mathcal{H}_n^d \, \kappa^2\!\left(\frac{n}{N}\right) \\
& \qquad \qquad \qquad \qquad \times \sum_{m=0}^{\min(K, n)}\sum_{k, k' \in \mathcal{I}_m^{d-1}} \zeta_{K, (m, k)}^{d, n} \, \overline{\zeta_{K, (m, k')}^{d, n}} \, t_{k', k}^{d-1, m}(h).
\end{align*}
For $d=3$ we have
\begin{equation*}
\langle T^3(h(\gamma))\Psi_{\scriptscriptstyle K}^{\scriptscriptstyle N} , \Psi_{\scriptscriptstyle K}^{\scriptscriptstyle N}  \rangle_{\mathbb{S}^{2}} = \sum_{n=0}^\infty \sum_{k=-\min(K, n)}^{\min(K, n)}  \dim \mathcal{H}_n^d \, \kappa^2\!\left(\frac{n}{N}\right) \lvert \zeta_{K,k }^{3, n} \rvert^2\, \mathrm{e}^{ \mathrm{i}k\gamma},
\end{equation*}
where $h(\gamma)$ is defined as in \hyperref[lemma2]{Lemma~\ref*{lemma2}}.
\end{proposition}
\noindent Note that due to the restrictions on the directionality components $\zeta_{K,k }^{d, n}$,  the auto-correlation function is always a polynomial of degree $K$ and does not change its structure when $N$ increases beyond $2K$. Hence, the wavelets $\Psi_{\scriptscriptstyle K}^{\scriptscriptstyle N}$ are limited in their directional sensitivity. To obtain wavelets with optimal directionality, the components $\zeta_{K, k}^{d, n}$ should be chosen such that the auto-correlation function is well localized at the identity. In \autoref{sec6}, the latter will be discussed explicitly for the case of highly symmetric frames.

We conclude this section by mentioning that our wavelets are steerable, a property which has been discussed previously in the two-dimen\-sional setting \citep{bib21}. It states that there exist fixed orientations $h_1, ..., h_M$ $\in SO(d-1)$ and polynomials $v_1, ..., v_M$ on $SO(d-1)$ such that
\begin{equation}\label{eq27}
T^d(h)\Psi_{\scriptscriptstyle K}^{\scriptscriptstyle N} = \sum_{p=1}^M v_p(h) \, T^d(h_p)\Psi_{\scriptscriptstyle K}^{\scriptscriptstyle N}, \quad \text{for all } h \in SO(d-1), \;N \in \mathbb{N}.
\end{equation}
To obtain a representation of the form \eqref{eq27} for $d>3$, one can start with some arbitrary quadrature rule satisfying
\begin{equation*}
\sum_{p=1}^M w_p f(h_p) = \int_{SO(d-1)} f \, \mathrm{d}\mu_{d-1} \quad \text{for all } f \in \Pi_{2K}(SO(d-1)).
\end{equation*}
The corresponding polynomials $v_p$, $p=1, ..., M$, can then be defined as
\begin{equation*}
v_p(h) = w_p \sum_{n=0}^K \dim \mathcal{H}_n^{d-1} \sum_{k, k' \in \mathcal{I}_n^{d-1}} \overline{t_{k, k'}^{d-1, n}(h_p)} \, t_{k, k'}^{d-1, n}(h)
\end{equation*}
and a straightforward calculation confirms the equality \eqref{eq27}. In particular, $v_1, ..., v_M \in \Pi_K(SO(d-1))$.

\section{Convergence in Banach spaces}\label{sec5}
Let the directional wavelets $\Psi_{\scriptscriptstyle K}^{j, \ell, m}$ be as constructed in \autoref{sec3} with $\kappa, \phi \in C^\infty([0, \infty))$. For $ 1\leq p \leq \infty$, we consider the family of operators $\Lambda_{J, \Omega} \colon L^p(\mathbb{S}^{d-1}) \rightarrow \Pi_{2^{J+1}}(\mathbb{S}^{d-1})$, $J \in \mathbb{N}_0$, defined by
\begin{align*}
\Lambda_{J, \Omega}f  = \sum_{j=0}^J \sum_{\ell=1}^{s_{j}} \sum_{m=1}^{r_j} \langle f,\Psi_{\scriptscriptstyle K}^{j, \ell, m}\rangle_{\mathbb{S}^{d-1}}\Psi_{\scriptscriptstyle K}^{j, \ell, m} + \sum_{(\ell, m)\in \Omega} \langle f,\Psi_{\scriptscriptstyle K}^{J+1, \ell, m}\rangle_{\mathbb{S}^{d-1}}\Psi_{\scriptscriptstyle K}^{J+1, \ell, m},
\end{align*}
where $\Omega \subset \{ 1, ..., s_{J+1}\} \times \{ 1, ..., r_{J+1} \}$. Such operators were previously studied for certain isotropic polynomial frames on $\mathbb{S}^{2}$ in \citep{bib29}.

\begin{lemma}\label{lemma5}
For $j \in \mathbb{N}$ it holds that
\begin{align}\label{eq4}
 \sum_{\ell=1}^{s_{j}} \sum_{m=1}^{r_j} \langle f,\Psi_{\scriptscriptstyle K}^{j, \ell, m}\rangle_{\mathbb{S}^{d-1}}\Psi_{\scriptscriptstyle K}^{j, \ell, m} = \sum_{n=0}^{\infty} \sum_{k \in  \mathcal{I}_{n}^d} \kappa^2\!\left(\frac{n}{2^{j-1}}\right) \langle f, Y_k^{d,n}\rangle_{\mathbb{S}^{d-1}} Y_k^{d,n}.
\end{align}
Furthermore, if $J \in \mathbb{N}_0$, we have
\begin{align*}
&\sum_{j=0}^J \sum_{\ell=1}^{s_{j}} \sum_{m=1}^{r_j} \langle f,\Psi_{\scriptscriptstyle K}^{j, \ell, m}\rangle_{\mathbb{S}^{d-1}}\Psi_{\scriptscriptstyle K}^{j, \ell, m} = \sum_{n=0}^{\infty}\sum_{k \in  \mathcal{I}_{n}^d} \phi^2\!\left(\frac{n}{2^J}\right) \langle f, Y_k^{d,n}\rangle_{\mathbb{S}^{d-1}} Y_k^{d,n}.
\end{align*}
\end{lemma}
\begin{proof}
For $j \in \mathbb{N}$, the definition of the wavelets $\Psi_{\scriptscriptstyle K}^{j, \ell, m}$ yields
\begin{align*}
& \sum_{\ell=1}^{s_{j}} \sum_{m=1}^{r_j} \langle f,\Psi_{\scriptscriptstyle K}^{j, \ell, m}\rangle_{\mathbb{S}^{d-1}}\Psi_{\scriptscriptstyle K}^{j, \ell, m}  \\
& \quad \qquad = \sum_{\ell=1}^{s_{j}} \sum_{m=1}^{r_j} \omega_{j, \ell} \, \mu_{j,m} \langle f, T^d(g_{\eta_{j, \ell}}h_{j,m}) \Psi_{\scriptscriptstyle K}^j \rangle_{\mathbb{S}^{d-1}} T^d(g_{\eta_{j, \ell}}h_{j,m}) \Psi_{\scriptscriptstyle K}^j.
\end{align*}
Thus, by \hyperref[lemma6]{Lemma~\ref*{lemma6}} we have
\begin{align*}
\sum_{\ell=1}^{s_{j}} \sum_{m=1}^{r_j} \langle f,\Psi_{\scriptscriptstyle K}^{j, \ell, m}\rangle_{\mathbb{S}^{d-1}}\Psi_{\scriptscriptstyle K}^{j, \ell, m}(\eta)  = \int_{SO(d)} \langle f, T^d(g) \Psi_{\scriptscriptstyle K}^j \rangle_{\mathbb{S}^{d-1}} T^d(g) \Psi_{\scriptscriptstyle K}^j(\eta) \, \mathrm{d}\mu_{d}(g)
\end{align*}
for every $\eta \in \mathbb{S}^{d-1}$. Equation \eqref{eq4} now follows easily from \eqref{eq102} and from the orthogonality relations of the matrix coefficients $t_{k, k'}^{d, n}$.

The second statement follows immediately from \eqref{eq4} due to \eqref{eqkappa}.
\end{proof}

In view of \hyperref[lemma5]{Lemma~\ref*{lemma5}}, the approximant $\Lambda_{J, \Omega}f$ can be written in the form
\begin{align*}
\Lambda_{J, \Omega}f =  \sum_{n=0}^{\infty}\sum_{k \in  \mathcal{I}_{n}^d} \phi^2\!\left(\frac{n}{2^J}\right) \langle f, Y_k^{d,n}\rangle_{\mathbb{S}^{d-1}} Y_k^{d,n}  + \sum_{(\ell, m)\in \Omega} \langle f,\Psi_{\scriptscriptstyle K}^{J+1, \ell, m}\rangle_{\mathbb{S}^{d-1}}\Psi_{\scriptscriptstyle K}^{J+1, \ell, m},
\end{align*}
or, even more compact,
\begin{equation}\label{eq104}
\Lambda_{J, \Omega}f = f\ast\Phi^J  + \sum_{(\ell, m)\in \Omega} \langle f,\Psi_{\scriptscriptstyle K}^{J+1, \ell, m}\rangle_{\mathbb{S}^{d-1}}\Psi_{\scriptscriptstyle K}^{J+1, \ell, m},
\end{equation}
where 
\begin{equation*}
f\ast \Phi^J(\eta) = \langle f, T^d(g_\eta)\Phi^J \rangle_{\mathbb{S}^{d-1}}
\end{equation*}
is the standard spherical convolution of $f \in L^1(\mathbb{S}^{d-1})$ with the isotropic scaling function $\Phi^J$ defined in \eqref{eq25}. Clearly,
\begin{equation*}
\Lambda_{J, \Omega}f = f \quad \text{for all } f \in \Pi_{2^{J-1}}(\mathbb{S}^{d-1}).
\end{equation*}

For the remainder of this section we will make the following assumption on the quadrature measure from \eqref{eq15}.
\begin{assumption}\label{assumption1}
There exist points $\tilde{h}_{j, m}\in SO(d-1)$ and positive weights $\tilde{\mu}_{j, m}$ such that
\begin{equation}\label{eq20}
\sum_{m =1}^{\tilde{r}_j} \tilde{\mu}_{j,m} \, \lvert f(\tilde{h}_{j, m}) \rvert \sim \int_{SO(d-1)} \lvert f \rvert \, \mathrm{d}\mu_{d-1}, \quad  f \in \mathcal{M}_{\min(K, 2^j)}^1(SO(d-1)),
\end{equation}
and
\begin{equation}\label{eq19}
\sum_{\ell=1}^{s_{j}}\sum_{m=1}^{\tilde{r}_j} \omega_{j, \ell} \, \tilde{\mu}_{j, m} \, \lvert f(g_{\eta_{j, \ell}} \tilde{h}_{j, m}) \rvert \leq c \int_{SO(d)} \lvert f \rvert \, \mathrm{d}\mu_{d}
\end{equation}
for any $f \in \mathcal{M}_{ 2^{j}, K}^1(SO(d))$. Here, the constants of equivalence in \eqref{eq20}, as well as the value $c>0$, are assumed to be independent of $j\in \mathbb{N}$.
\end{assumption}

\begin{lemma}
We have
\begin{equation}\label{eq16}
\sum_{\ell=1}^{s_{j}}\sum_{m=1}^{r_j} \omega_{j, \ell} \, \mu_{j, m} \, \lvert f(g_{\eta_{j, \ell}} h_{j, m}) \rvert \leq c \int_{SO(d)} \lvert f \rvert \, \mathrm{d}\mu_{d}
\end{equation}
for every $f \in \mathcal{M}_{2^j, K}(SO(d))$.  Here, $c>0$ is independent of $j\in \mathbb{N}$.
\end{lemma}
\begin{proof}
Since $h \mapsto f(g_{\eta_{j, \ell}} h)$ is a function in $\mathcal{M}_{\min(K, 2^j)}(SO(d-1))$ by \hyperref[lemma2.3]{Proposition~\ref*{lemma2.3}},
\begin{equation*}
\sum_{m=1}^{r_j} \mu_{j, m} \, \lvert f(g_{\eta_{j, \ell}} h_{j, m}) \rvert\leq c \int_{SO(d-1)} \lvert f(g_{\eta_{j, \ell}} h)\rvert \, \mathrm{d}\mu_{d-1}(h)
\end{equation*}
holds trivially. Additionally, \eqref{eq20} yields
\begin{equation*}
\int_{SO(d-1)} \lvert f(g_{\eta_{j, \ell}} h)\rvert \, \mathrm{d}\mu_{d-1}(h) \leq c \sum_{m =1}^{\tilde{r}_j} \tilde{\mu}_{j,m} \, \lvert f(g_{\eta_{j, \ell}} \tilde{h}_{j, m}) \rvert.
\end{equation*}
Hence, combining the above inequalities and using \eqref{eq19} we get
\begin{align*}
\sum_{\ell=1}^{s_{j}}\sum_{m=1}^{r_j} \omega_{j, \ell} \, \mu_{j, m} \, \lvert f(g_{\eta_{j, \ell}} h_{j, m}) \rvert \leq \sum_{\ell=1}^{s_{j}}\sum_{m=1}^{\tilde{r}_j} \omega_{j, \ell} \, \tilde{\mu}_{j, m} \, \lvert f(g_{\eta_{j, \ell}} \tilde{h}_{j, m}) \rvert \leq c \int_{SO(d)} \lvert f \rvert \, \mathrm{d}\mu_{d}.
\end{align*}
\end{proof}

\begin{proposition}\label{prop1}
For $1\leq p \leq \infty$ it holds that 
\begin{equation*}
\|\Lambda_{J, \Omega}f \|_{ L^p(\mathbb{S}^{d-1})} \leq c_p \|f \|_{L^p(\mathbb{S}^{d-1})}.
\end{equation*}
\end{proposition}
\begin{proof}
It is well known, that $\| f \ast \Phi^J\|_{L^p(\mathbb{S}^{d-1})} \leq c_p \|f \|_{L^p(\mathbb{S}^{d-1})}$. For a proof of the latter, we refer to \citep[Theorem 2.6.3]{bib3}. Consequently, by \eqref{eq104} it suffices to show that
\begin{equation*}
\left \| \sum_{(\ell, m)\in \Omega} \langle f,\Psi_{\scriptscriptstyle K}^{J+1, \ell, m}\rangle_{\mathbb{S}^{d-1}}\Psi_{\scriptscriptstyle K}^{J+1, \ell, m} \right\|_{L^p(\mathbb{S}^{d-1})} \leq c_p \|f \|_{L^p(\mathbb{S}^{d-1})}.
\end{equation*}
For this, we will use of the equality
\begin{align}\label{eq9}
&\sum_{(\ell, m)\in \Omega} \langle f,\Psi_{\scriptscriptstyle K}^{J+1, \ell, m}\rangle_{\mathbb{S}^{d-1}}\Psi_{\scriptscriptstyle K}^{J+1, \ell, m}(\eta) \nonumber \\
& = \int_{\mathbb{S}^{d-1}} f(\nu) \sum_{(\ell, m)\in \Omega} \overline{\Psi_{\scriptscriptstyle K}^{J+1, \ell, m}(\nu)}\,  \Psi_{\scriptscriptstyle K}^{J+1, \ell, m}(\eta) \, \mathrm{d}\omega_{d-1}(\nu), \quad \eta \in \mathbb{S}^{d-1}.
\end{align}
\textbf{$p=1$:} By \eqref{eq9}, we have
\begin{align*}
& \left \| \sum_{(\ell, m)\in \Omega} \langle f,\Psi_{\scriptscriptstyle K}^{J+1, \ell, m}\rangle_{\mathbb{S}^{d-1}}\Psi_{\scriptscriptstyle K}^{J+1, \ell, m} \right\|_{L^1(\mathbb{S}^{d-1})} \leq \\
&  \qquad \int_{\mathbb{S}^{d-1}} \lvert f(\nu) \rvert \int_{\mathbb{S}^{d-1}} \sum_{(\ell, m)\in \Omega} \lvert \Psi_{\scriptscriptstyle K}^{J+1, \ell, m}(\nu) \, \Psi_{\scriptscriptstyle K}^{J+1, \ell, m}(\eta) \rvert \, \mathrm{d}\omega_{d-1}(\eta) \, \mathrm{d}\omega_{d-1}(\nu).
\end{align*}
Furthermore, 
\begin{align*}
 &\sum_{(\ell, m)\in \Omega} \lvert \Psi_{\scriptscriptstyle K}^{J+1, \ell, m}(\nu) \, \Psi_{\scriptscriptstyle K}^{J+1, \ell, m}(\eta) \rvert\\
& \qquad  \leq \sum_{\ell =1}^{s_{J+1}} \sum_{m=1}^{r_{J+1}} \omega_{J+1, \ell}\, \mu_{J+1, m} \,\lvert T^d(g_{\eta_{J+1, \ell}} h_{J+1, m})\Psi_{\scriptscriptstyle K}^{J+1}(\nu) \, T^d(g_{\eta_{J+1, \ell}} h_{J+1, m})\Psi_{\scriptscriptstyle K}^{J+1}(\eta)\rvert.
\end{align*}
By \hyperref[corollary1]{Corollary~\ref*{corollary1}},
\begin{equation*}
\int_{\mathbb{S}^{d-1}} \lvert T^d(g_{\eta_{J+1, \ell}} h_{J+1, m})\Psi_{\scriptscriptstyle K}^{J+1}(\eta)\rvert \, \mathrm{d}\omega_{d-1}(\eta) \leq c.
\end{equation*}
Additionally, it follows from \eqref{eq16} that
\begin{align*}
& \sum_{\ell =1}^{s_{J+1}} \sum_{m=1}^{r_{J+1}} \omega_{J+1, \ell}\, \mu_{J+1, m} \,\lvert T^d(g_{\eta_{J+1, \ell}} h_{J+1, m})\Psi_{\scriptscriptstyle K}^{J+1}(\nu)\rvert\\
& \qquad \qquad  \leq c \int_{SO(d)}  \lvert T^d(g)\Psi_{\scriptscriptstyle K}^{J+1}(\nu)\rvert \, \mathrm{d}\mu_{d}(g) \leq c.
\end{align*}
\textbf{$p=\infty$:} Starting again from \eqref{eq9} and using similar arguments as before, we obtain
\begin{equation*}
\bigg\lvert \sum_{(\ell, m)\in \Omega} \langle f,\Psi_{\scriptscriptstyle K}^{J+1, \ell, m}\rangle_{\mathbb{S}^{d-1}}\Psi_{\scriptscriptstyle K}^{J+1, \ell, m}(\eta) \bigg\rvert \leq c \|f\|_{L^\infty(\mathbb{S}^{d-1})}.
\end{equation*}
\textbf{$1<p<\infty$:} In this case, the statement follows from the Riesz-Thorin interpolation theorem.
\end{proof}
In the following we will use the notation
\begin{equation*}
E_m(f)_p = \inf_{P \in \Pi_m(\mathbb{S}^{d-1})} \| f-P \|_{L^p(\mathbb{S}^{d-1})}
\end{equation*}
for $f \in L^p(\mathbb{S}^{d-1})$, $1\leq p \leq \infty$. A classical argument shows that any sequence $(\Lambda_{j, \Omega_j}f)_{j=0}^\infty$ with $\Omega_j \subset \{1, ...,s_{j+1} \} \times \{1, ..., r_{j+1} \}$ approximates $f$ with the rate of the best approximation.

\begin{corollary}
Let $f \in L^p(\mathbb{S}^{d-1})$ with $1 \leq p \leq \infty$. Then
\begin{equation*}
\| f- \Lambda_{j, \Omega}f \|_{L^p(\mathbb{S}^{d-1})} \leq c_p E_{2^{j-1}}(f)_p
\end{equation*}
for any $\Omega \subset \{1, ...,s_{j+1} \} \times \{1, ..., r_{j+1} \}$.
\end{corollary}
\begin{proof}
Let $P$ be the best approximation of $f$ in $\Pi_{2^{j-1}}(\mathbb{S}^{d-1})$ with respect to $\| \cdot \|_{L^p(\mathbb{S}^{d-1})}$, i.e., $\|f-P\|_{L^p(\mathbb{S}^{d-1})} = E_{2^{j-1}}(f)_p$. Then $\Lambda_{j, \Omega} P = P$ and \hyperref[prop1]{Proposition~\ref*{prop1}} yields
\begin{equation*}
\| f- \Lambda_{j, \Omega}f \|_{L^p(\mathbb{S}^{d-1})} = \| f- P + \Lambda_{j, \Omega}(P-f) \|_{L^p(\mathbb{S}^{d-1})} \leq c_p E_{2^{j-1}}(f)_p.
\end{equation*}
\end{proof}

\section{Wavelets with additional symmetries}\label{sec6}
To conclude this article, we will discuss important special cases in which the wavelets constructed exhibit a large amount of symmetries. Indeed, if $m \in \{2, 3, .., d-1\}$ and the directionality components $\zeta_{ K, k}^{d,n}$ vanish whenever $k_{d-m}>0$, then the resulting wavelets $\Psi_{\scriptscriptstyle K}^{j}$ in \eqref{waveletdef} are rotationally invarient with respect to $SO(m)$, i.e., we have $T^d(h)\Psi_{\scriptscriptstyle K}^{j} = \Psi_{\scriptscriptstyle K}^{j}$ for each $h \in SO(m)$. This can easily be seen by starting from the representation in \eqref{eq102} and repeatedly applying \hyperref[lemma2]{Lemma~\ref*{lemma2}} to obtain
\begin{align*}
& T^d(h)\Psi_{\scriptscriptstyle K}^{j} = \sum_{n=0}^\infty \sum_{k, k'\in \mathcal{I}_n^d} \sqrt{\dim \mathcal{H}_n^d} \, \kappa\!\left(\frac{n}{2^{j-1}}\right) \zeta_{ K, k'}^{d,n} \\
& \qquad \qquad \qquad \times \delta_{(k_1, ..., k_{d-m}), (k_1', ..., k_{d-m}')} \, t_{(k_{d-m+1}, ...), (k_{d-m+1}', ...)}^{m, k_{d-m}}(h) \, Y_k^{d, n}
\end{align*}
for $h \in SO(m)$. To further investigate different variants of symmetric wavelets, we introduce the polynomial subspaces
\begin{align*}
&\mathcal{M}_{N, M}^1(SO(d)/SO(m)) \\
& \qquad = \{f \in \mathcal{M}_{N, M}^1(SO(d)) \mid f(g h) = f(g) \text{ for all } g \in SO(d), \; h \in SO(m) \}.
\end{align*}
Clearly, the symmetry
\begin{equation*}
 T^d(h)\Psi_{\scriptscriptstyle K}^{j} = \Psi_{\scriptscriptstyle K}^{j} \qquad \text{if } h \in SO(m)
\end{equation*}
implies that
\begin{equation*}
g \mapsto \langle f, T^d(g)\Psi_{\scriptscriptstyle K}^{j} \rangle_{\mathbb{S}^{d-1}} \in \mathcal{M}_{2^j, K}^1(SO(d)/SO(m))
\end{equation*}
for all $f \in L^1(\mathbb{S}^{d-1})$. Thus, whenever we have quadrature points $g_{j, \ell} \in SO(d)$, $\ell = 1, ..., s_j$, with non-negative weights $\mu_{j, \ell}$ such that
\begin{equation*}
\int_{SO(d)} \lvert  P(g) \rvert^2 \, \mathrm{d}\mu_{d}(g) = \sum_{\ell =1}^{s_j} \mu_{j, \ell}\, \lvert P(g_{j, \ell})\rvert^2
\end{equation*}
if $P \in \mathcal{M}_{2^j, K}^1(SO(d)/SO(m))$, then the wavelets
\begin{equation*}
\Psi_{\scriptscriptstyle K}^{ 0, 1} \equiv  1, \quad \Psi_{\scriptscriptstyle K}^{j, \ell} = \sqrt{\mu_{j,\ell}} \, T^d( g_{j,\ell})\Psi_{\scriptscriptstyle K}^{j}, \qquad  \ell = 1, ..., s_j, \quad j \geq 1,
\end{equation*}
form a tight frame for $L^2(\mathbb{S}^{d-1})$. In particular, compared to the general case discussed in \autoref{sec3}, the amount of quadrature points needed for highly symmetric wavelets can be reduced significantly. In what follows, the two most symmetric cases $m=d-1$ and $m=d-2$ are layed out in detail. Apart from discussing numerical integration, we also give an explicit construction of directionality components which yield wavelets of optimal directionality.

\subsection{The zonal case}\label{subsec6}
We first consider the case $m=d-1$, where we assume the symmetry
\begin{equation*}
T^d(h)\Psi_{\scriptscriptstyle K}^{j} = \Psi_{\scriptscriptstyle K}^{j} \qquad \text{for all } h \in SO(d-1), \; j\in \mathbb{N}_0.
\end{equation*}
This corresponds to the setting of isotropic, or zonal, wavelets which are obtained from our construction by choosing $K=0$. Here, every $P\in \mathcal{M}_{N, 0}^1(SO(d)/SO(d-1))$ can be identified with a polynomial in $\Pi_N(\mathbb{S}^{d-1})$ by \hyperref[lemma4]{Lemma~\ref*{lemma4}}, since
\begin{equation*}
P(g_\eta h) =P(g_\eta) = \int_{SO(d-1)} P(g_\eta h') \, \mathrm{d}\mu_{d-1}(h')
\end{equation*}
for all $h \in SO(d-1)$. Hence, if for each $j \in \mathbb{N}$ we have quadrature points $\eta_{j, \ell}\in \mathbb{S}^{d-1}$ with positive weights $\omega_{j, \ell}$ such that
\begin{equation*}
\int_{\mathbb{S}^{d-1}}f \,\mathrm{d}\omega_{d-1} = \sum_{\ell =1}^{s_j} \omega_{j, \ell} \,f(\eta_{j, \ell})\quad \text{for all } f \in \Pi_{2^{j+1}}(\mathbb{S}^{d-1}),
\end{equation*}
then the isotropic wavelets
\begin{equation*}
\Psi_{\scriptscriptstyle K}^{ 0, 1} \equiv  1, \quad \Psi_{\scriptscriptstyle K}^{j, \ell} = \sqrt{\omega_{j,\ell}} \, T^d( g_{\eta_{j,\ell}})\Psi_{\scriptscriptstyle K}^{j}, \qquad  \ell = 1, ..., s_j, \quad j \geq 1,
\end{equation*}
form a tight frame for $L^2(\mathbb{S}^{d-1})$.

\subsection{Wavelets of high directional sensitivity}

Let us now consider the case $m=d-2$ with $d>3$. I.e., we assume the rotational invariance
\begin{equation}\label{eq105}
T^d(h)\Psi_{\scriptscriptstyle K}^{j} = \Psi_{\scriptscriptstyle K}^{j} \qquad \text{for all } h \in SO(d-2), \; j \in \mathbb{N}_0,
\end{equation}
which is obtained by choosing $\zeta_{ K, k}^{d,n}=0$ whenever $k_{2}>0$.  Utilizing \eqref{eq2}, we have
\begin{align*}
&\int_{SO(d)} \lvert \langle f, T^d(g)\Psi_{\scriptscriptstyle K}^{j} \rangle_{\mathbb{S}^{d-1}} \rvert^2 \, \mathrm{d}\mu_d(g) \\
&  \qquad \qquad = \int_{\mathbb{S}^{d-1}} \int_{\mathbb{S}^{d-2}}  \lvert \langle f, T^d(g_{\eta} g_{\hat{\eta}})\Psi_{\scriptscriptstyle K}^{j} \rangle_{\mathbb{S}^{d-1}} \rvert^2 \, \mathrm{d}\omega_{d-2}(\hat{\eta}) \, \mathrm{d}\omega_{d-1}(\eta).
\end{align*}
As discussed in \autoref{sec3}, the map $h \mapsto \langle f, T^d(h)\Psi_{\scriptscriptstyle K}^{j}\rangle_{\mathbb{S}^{d-1}}$ belongs to the polynomial subspace $\mathcal{M}_{\min(2^j,K)}^1(SO(d-1))$. Thus, 
\begin{equation*}
\hat{\eta} \mapsto \langle f, T^d(g_{\eta} g_{\hat{\eta}})\Psi_{\scriptscriptstyle K}^{j} \rangle_{\mathbb{S}^{d-1}} 
\end{equation*}
is contained in $\Pi_{\min(2^j, K)}(\mathbb{S}^{d-2})$ by \hyperref[lemma4]{Lemma~\ref*{lemma4}}, since
\begin{align*}
\langle f, T^d(g_{\eta} g_{\hat{\eta}})\Psi_{\scriptscriptstyle K}^{j} \rangle_{\mathbb{S}^{d-1}} = \int_{SO(d-2)} \langle f, T^d(g_{\eta} g_{\hat{\eta}} h)\Psi_{\scriptscriptstyle K}^{j} \rangle_{\mathbb{S}^{d-1}} \, \mathrm{d}\mu_{d-2}(h).
\end{align*}
Additionally, \hyperref[lemma4]{Lemma~\ref*{lemma4}} shows that
\begin{align*}
\eta \mapsto  \int_{\mathbb{S}^{d-2}}  \lvert \langle f, T^d(g_{\eta} g_{\hat{\eta}})\Psi_{\scriptscriptstyle K}^{j} \rangle_{\mathbb{S}^{d-1}} \rvert^2 \, \mathrm{d}\omega_{d-2}(\hat{\eta})
\end{align*}
is in $\Pi_{2^{j+1}}(\mathbb{S}^{d-1})$. Now let $\eta_{j, \ell} \in \mathbb{S}^{d-1}$, $\ell = 1, ... , s_j$, be points with positive weights $\omega_{ j, \ell}$ such that
\begin{equation*}
\int_{\mathbb{S}^{d-1}} P \, \mathrm{d}\mu_{d-1} = \sum_{\ell=1}^{s_j} \omega_{ j, \ell} \, P(\eta_{j, \ell}) \quad \text{for each } P \in \Pi_{2^{j+1}}(\mathbb{S}^{d-1}).
\end{equation*}
Similarly, we assume that
\begin{equation*}
\int_{\mathbb{S}^{d-2}} P \, \mathrm{d}\mu_{d-2} = \sum_{m=1}^{r_j} \hat{\omega}_{ j, m} \, P(\hat{\eta}_{j, m}) \quad \text{if } P \in \Pi_{2\min(2^j, K)}(\mathbb{S}^{d-2})
\end{equation*}
with $\hat{\eta}_{ j, m}\in \mathbb{S}^{d-2}$ and $\hat{\omega}_{ j, m} >0$. Then
\begin{align*}
\int_{SO(d)} \lvert \langle f, T^d(g)\Psi_{\scriptscriptstyle K}^{j} \rangle_{\mathbb{S}^{d-1}} \rvert^2 \, \mathrm{d}\mu_d(g) = \sum_{\ell=1}^{s_j} \sum_{m=1}^{r_j} \omega_{ j, \ell}\, \hat{\omega}_{ j, m} \, \lvert \langle f, T^d(g_{\eta_{ j, \ell}} g_{\hat{\eta}_{ j, m}})\Psi_{\scriptscriptstyle K}^{j} \rangle_{\mathbb{S}^{d-1}} \rvert^2.
\end{align*}
It follows that the wavelets
\begin{equation*}
\Psi_{\scriptscriptstyle K}^{ 0, 1,1} \equiv  1, \quad \Psi_{\scriptscriptstyle K}^{j, \ell, m} = \sqrt{ \omega_{j, \ell}\, \hat{\omega}_{j, m}} \, T^d( g_{\eta_{ j, \ell}} g_{\hat{\eta}_{ j, m}})\Psi_{\scriptscriptstyle K}^{j},
\end{equation*}
$\ell = 1, ..., s_j$, $m=1, ..., r_j$, $j \geq 1$, form a tight frame for $L^2(\mathbb{S}^{d-1})$.

In the following, we will propose specific choices for the values $ \zeta_{K, k}^{d, n}$ in \eqref{eq11} which yield highly directional wavelets. Hence, the functions $\Psi_{\scriptscriptstyle K}^{\scriptscriptstyle N}$, we again adopt the notation from \autoref{sec4}, constructed in this section should be well suited for analyzing anisotropic local structures such as edges. Let
\begin{equation*}
\zeta_{K, k}^{d, n} = \delta_{k_2, 0}\, \phi_{{\scriptscriptstyle K}, n}(k_1)
\end{equation*}
for some complex-valued functions $\phi_{{\scriptscriptstyle K}, n}$ on $\{0, 1, ..., n \}$. This yields wavelets of the form
\begin{align}\label{eq28}
\Psi_{\scriptscriptstyle K}^{\scriptscriptstyle N}(x) &= \sum_{n=0}^\infty \sum_{m=0}^n \sqrt{\dim \mathcal{H}_n^d }\, \kappa\!\left(\frac{n}{N}\right) \phi_{{\scriptscriptstyle K}, n}(m) \, A_{(m, 0, ..., 0)}^n \nonumber \\
& \qquad \qquad \qquad \times  C_{n-m}^{\frac{d-2}{2}+m}(x_d) \, C_m^{\frac{d-3}{2}}\!\left(\frac{x_{d-1}}{\sqrt{1-x_d^2}}\right) (1-x_d^2)^{m/2}
\end{align}
for $x= (x_1, x_2, ..., x_d) \in \mathbb{S}^{d-1}$. To ensure that the assumptions that we made on $\zeta_{K, k}^{d, n}$ in \autoref{sec3} are satisfied, each $ \phi_{{\scriptscriptstyle K}, n}$ has to be supported in $\{0, 1, ..., K\}$. Furthermore, $ \phi_{{\scriptscriptstyle K}, n}$ must be independent of $n$ for $n\geq K$, in which case we write $\phi_{{\scriptscriptstyle K}, n} = \phi_{{\scriptscriptstyle K}}$, and
\begin{equation*}
\sum_{m=0}^{\min(K, n)} \lvert \phi_{{\scriptscriptstyle K}, n}(m) \rvert^2 = 1
\end{equation*}
must hold for each $n \in \mathbb{N}$. 

Since the wavelets $\Psi_{\scriptscriptstyle K}^{\scriptscriptstyle N}$ in \eqref{eq28} are highly symmetric, there is a convenient way to visualize them in terms of two-dimensional functions. We have shown in \autoref{theorem1} that $\Psi_{\scriptscriptstyle K}^{\scriptscriptstyle N}$ is concentrated at the north pole $e^d$. Clearly, each element $\xi$ in the unit tangent space $UT_{e^d}\mathbb{S}^{d-1} = \{ (\eta, 0) \mid \eta \in \mathbb{S}^{d-2} \}$, which we simply identify with $\mathbb{S}^{d-2}$, corresponds to a smooth curve of the form
\begin{equation*}
t \mapsto \cos t \,  e^d + \sin t \, \xi, \qquad t \in [0,\pi],
\end{equation*}
and we want to illustrate the behavior of $\Psi_{\scriptscriptstyle K}^{\scriptscriptstyle N}$ along each of these paths. 
Moreover, parameterizing $\mathbb{S}^{d-2}$ by $[0, 2\pi) \times \mathbb{S}^{d-3}$ via $(\varphi, \eta_{d-3})\mapsto  \cos \varphi \, e^{d-1} +  \sin\varphi \, \eta_{d-3}$, we are lead to define
\begin{align*}
& \psi_{\scriptscriptstyle K}^{\scriptscriptstyle N}(t, \varphi)  = \Psi_{\scriptscriptstyle K}^{\scriptscriptstyle N}(\cos t \,  e^d + \sin t \, (\cos \varphi \, e^{d-1} + \sin \varphi \, \eta_{d-3}))\nonumber \\
& \quad =  \sum_{n=0}^\infty \sum_{m=0}^n \sqrt{\dim \mathcal{H}_n^d} \, \kappa\!\left( \frac{n}{N} \right)  \phi_{{\scriptscriptstyle K}, n}(m) \, A_{(m, 0, ..., 0)}^n C_{n-m}^{\frac{d-2}{2}+m}(\cos t) \, C_m^{\frac{d-3}{2}}(\cos \varphi) \, \sin^m t
\end{align*}
for $t \in [0, \pi], \; \varphi \in [0, 2\pi)$, as the expression on the right hand sight does not depend on $\eta_{d-3} \in \mathbb{S}^{d-3}$ according to \eqref{eq28}. The latter implies that for any fixed $\varphi$, $\Psi_{\scriptscriptstyle K}^{\scriptscriptstyle N}$ exhibits the same behavior in all directions $\cos \varphi e^{d-1}+ \sin \varphi \,  \eta_{d-3}$, $\eta_{d-3} \in \mathbb{S}^{d-3}$. Now, $\psi_{\scriptscriptstyle K}^{\scriptscriptstyle N}$ can be interpreted as a function on $\mathbb{B}_\pi^2= \{ x \in \mathbb{R}^2 \mid \| x\| \leq \pi \}$ given in terms of polar coordinates $(t, \varphi)$. 

The wavelets designed in the following exhibit the symmetry
\begin{equation}\label{eq107}
\Psi_{\scriptscriptstyle K}^{\scriptscriptstyle N}(\cos t \, e^d + \sin t \, (- \xi)) = (-1)^K \, \Psi_{\scriptscriptstyle K}^{\scriptscriptstyle N}(\cos t \,  e^d + \sin t \, \xi)
\end{equation}
for $\xi \in UT_{e^d}\mathbb{S}^{d-1}$, $ t \in [0, \pi]$. This is equivalent to 
\begin{equation*}
\psi_{\scriptscriptstyle K}^{\scriptscriptstyle N}(t, \varphi) = (-1)^K \psi_{\scriptscriptstyle K}^{\scriptscriptstyle N}(t, \varphi+ \pi)
\end{equation*}
and thus can be achieved by setting $ \phi_{{\scriptscriptstyle K}, n}(m) = 0$ whenever $K-m$ is odd, since $C_m^\lambda(-t) = (-1)^m C_m^\lambda (t)$.
\begin{figure}
\includegraphics[width=\textwidth]{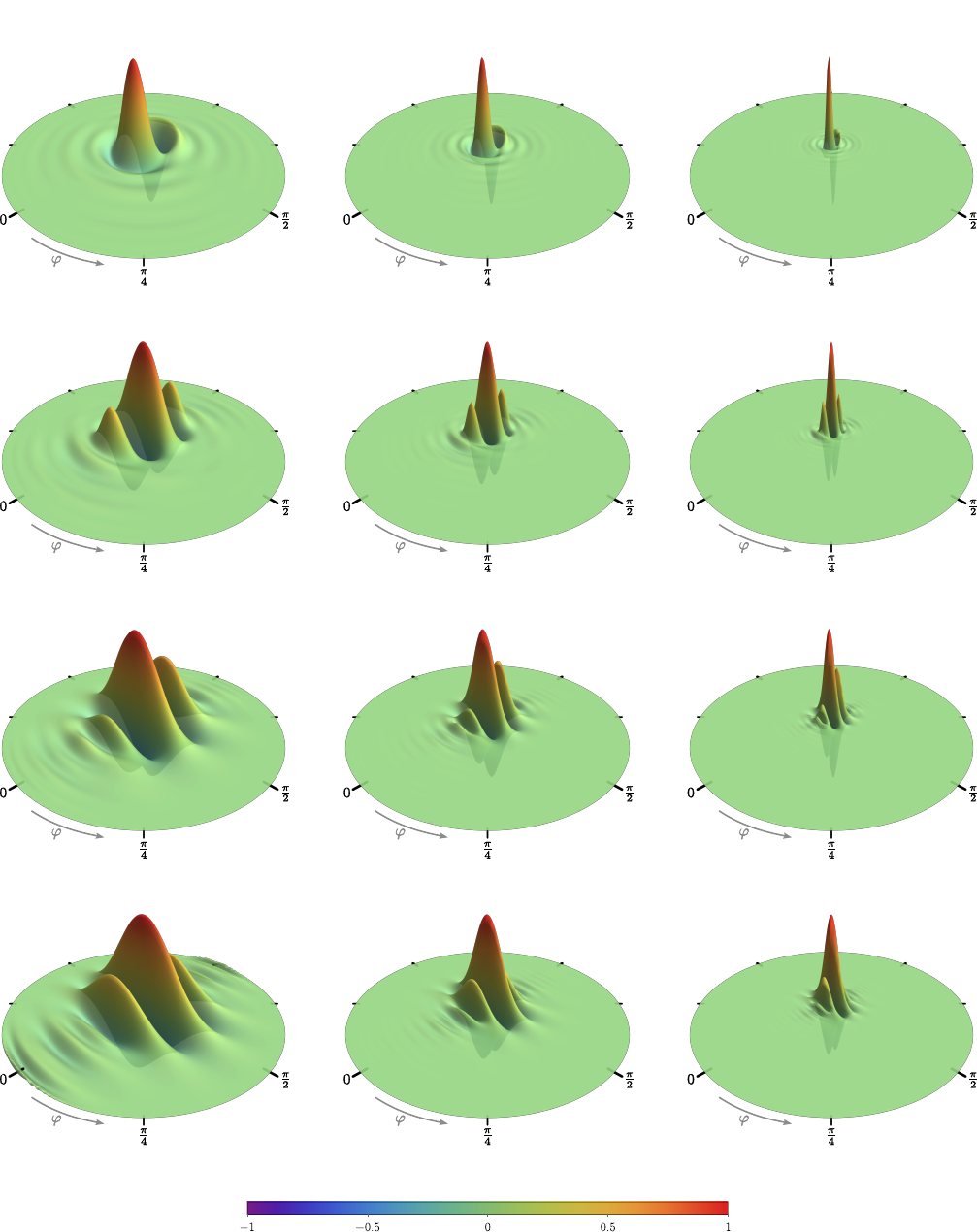}
\caption{Re-scaled functions $\psi_{\scriptscriptstyle K}^{\scriptscriptstyle N} (t, \varphi)$, $(t, \varphi) \in [0, \pi/2]\times[0, 2\pi)$, for $K=1, 4,9,16$ from top to bottom and $N=16, 32, 64$ from left to right}\label{fig1}
\end{figure}

Up to now, we discussed how wavelets with a large amount of symmetries, namely \eqref{eq105} and \eqref{eq107}, can be obtained. It remains the question of how the values $ \phi_{{\scriptscriptstyle K}, n}(m)$ might be chosen in order to guarantee an optimal directionality. As discussed in \autoref{sec4}, we will tackle this problem by optimizing the localization of the corresponding auto-correlation function. Using \hyperref[prop3.2]{Proposition~\ref*{prop3.2}}, it is not hard to arrive at the formula
\begin{align*}
&\langle T^d(h) \Psi_{\scriptscriptstyle K}^{\scriptscriptstyle N},  \Psi_{\scriptscriptstyle K}^{\scriptscriptstyle N}\rangle_{\mathbb{S}^{d-1}}= \sum_{n=0}^{\infty}\dim  \mathcal{H}_n^d \, \kappa^2\! \left( \frac{n}{N} \right) \nonumber \\
&\qquad \qquad \times \sum_{m=0}^{\min(K, n)}\frac{\Gamma(d-3) m!}{\Gamma(d+m-3)} \lvert\phi_{{\scriptscriptstyle K}, n}(m) \rvert^2 C_m^{\frac{d-3}{2}}(\langle e^{d-1} , h e^{d-1}\rangle),
\end{align*}
$h \in SO(d-1)$. Since the above expression only depends on $he^{d-1}$, the auto-correlation function can be viewed as a map on $\mathbb{S}^{d-2}$. Due to the fact that $\kappa(0)=0$, we can simply define $\phi_{{\scriptscriptstyle K}, 0}\equiv 0$. Otherwise, for $K_n = \min(K, n)$ let
\begin{equation*}
\phi_{{{\scriptscriptstyle K}, n}}(m) = \sqrt{\frac{\Gamma(\lambda)}{\Gamma(2\lambda)}} \begin{cases}
 \tilde{\phi}_{K_n}(m), \quad & m \in \{0, 1, ..., K_n\},\\
0 & \text{else},
\end{cases}
\end{equation*}
where
\begin{equation*}
\tilde{\phi}_{K_n}(m) = \begin{cases}
(-1)^{\lfloor \frac{m}{2} \rfloor}  \sqrt{ \frac{K_n !(m+\lambda) \Gamma(d+m-3)}{2^{K_n}(\frac{K_n-m}{2})! \Gamma(\lambda +\frac{K_n+m}{2}+1) m!}}, \quad & K_n -m \text{ even}, \\
0, & K_n -m \text{ odd},
\end{cases}
\end{equation*}
and $\lambda = \frac{d-3}{2}$. Note that the conditions on $ \phi_{{\scriptscriptstyle K}, n}$ discussed above are satisfied due to the well known identity
\begin{equation*}
t^n = \frac{n!}{2^n} \sum_{\substack{m=0 \\ n-m \text{ even}}}^n \frac{(m+ \lambda ) \Gamma(\lambda)}{(\frac{n-m}{2})! \Gamma(\lambda + \frac{n+m}{2} +1)} C_m^\lambda(t),\quad t \in [-1,1], \; n \in \mathbb{N}_0.
\end{equation*}
This formula (see e.g.\ \citep[p. 487]{bib32}) also shows that
\begin{equation*}
\langle T^d(h) \Psi_{\scriptscriptstyle K}^{\scriptscriptstyle N},  \Psi_{\scriptscriptstyle K}^{\scriptscriptstyle N}\rangle_{\mathbb{S}^{d-1}}= \sum_{n=0}^{\infty}\dim  \mathcal{H}_n^d \, \kappa^2\! \left( \frac{n}{N} \right) (\langle e^{d-1} , h e^{d-1}\rangle)^{K_n}.
\end{equation*}
Hence, the auto-correlation function is of optimal localization under the constraints set by $K$. The resulting wavelets $\Psi_{\scriptscriptstyle K}^{\scriptscriptstyle N}$ are visualized in terms of the functions $\psi_{\scriptscriptstyle K}^{\scriptscriptstyle N}$ in \autoref{fig1}. Here, for the sake of clarity, each of the images has been re-scaled to take values between $-1$ and $1$.

\bibliographystyle{elsarticle-harv}
\bibliography{mybib}

\end{document}